\documentclass{amsart}
\usepackage{amsmath}
\usepackage{amssymb}
\usepackage{latexsym}
\usepackage{amsbsy}
\usepackage[all]{xy}
\usepackage{amscd}
\usepackage[dvips]{graphicx}
\newtheorem{theorem}{Theorem}[section]
\newtheorem{lemma}[theorem]{Lemma}

\theoremstyle{definition}

\newtheorem{Prop}[theorem]{Proposition}
\newtheorem{Cor}[theorem]{Corollary}

\theoremstyle{remark}
\newtheorem{remark}[theorem]{Remark}

\numberwithin{equation}{section}

\newcommand{\str}[1]{\langle #1\rangle}
\def\az{\alpha}      \def\ud{{\underline{d}}}

\def\bbn{{\mathbb N}}  \def\bbz{{\mathbb Z}}  \def\bbq{{\mathbb Q}} \def\bb1{{\mathbb 1}}
   \def\bbe{{\mathbb E}} 

  \def\bbc{{\mathbb C}}

\def\ra{\rightarrow}
\def\hom{\mbox{Hom}}

 \def\ind{\mbox{ind}\,}

\def\ext{\mbox{Ext}\,} 
\def\dim{\mbox{dim}\,}

\def\mod{\mbox{mod}\,}

\def\uq2{U_q(\hat{sl}_2)}

\def\bb{{\bf b}}

\def\nd{{\noindent}}

\def\mc{{\mathcal{C}}}

\def\md{{\mathcal{D}}}

\def\mo{{\mathcal{O}}}

\def\mh{{\mathcal{H}}}

\begin{document}

\title[On the cluster multiplication Theorem]{On the cluster multiplication theorem for acyclic cluster algebras}

\author{Fan Xu}
\address{Department of Mathematical Sciences, Tsinghua University, Beijing 100084, P.R.China}

\email{fanxu@mail.tsinghua.edu.cn}
\thanks{The research was partially
supported  by NSF of China (No. 10631010)}

\subjclass[2000]{Primary  16G20, 16G70; Secondary  14M99, 18E30}

\date{November 15, 2007.}

\keywords{2-Calabi-Yau, Cluster category}

\begin{abstract}
In \cite{CK2005} and \cite{Hubery2005}, the authors proved the
cluster multiplication theorems for finite type and affine type.
We generalize their results and prove the cluster multiplication
theorem for arbitrary type by using the properties of
2--Calabi--Yau
 (Auslander--Reiten formula) and high order associativity.
\end{abstract}

\maketitle

\section*{Introduction}
Cluster algebras were introduced by Fomin and Zelevinsky in
\cite{FZ2002}. By definition, the cluster algebras are commutative
algebras generated by a set of variables called cluster variables.
Let $Q$ be a quiver and we denote by $\mathcal{A}(Q)$ the
associated cluster algebra.  If $Q$ does not contain oriented
cycles, we call $Q$ an acyclic quiver. The cluster algebras
associated to acyclic quivers  are called acyclic cluster
algebras. Their relations to quiver representations were first
revealed in \cite{MRZ}. In~\cite{BMRRT}, the authors found the
general framework of the link of cluster algebras and quiver
representations and introduced the cluster categories as the
categorification of acyclic cluster algebras. For an acyclic
quiver $Q$, the associated cluster category $\mathcal{C}(Q)$ is
the orbit category of the bounded derived category
$\mathcal{D}^{b}(\mathrm{mod}kQ)$ over a field $k$ by the
auto--equivalence $F:=[1]\tau^{-1}$ where $[1]$ is the translation
functor and $\tau$ is the AR-translation. In general, one can
define the cluster category of a hereditary category with Serre
duality $\nu$ by taking $\tau=[-1]\nu$ as shown in \cite{Keller}.

In \cite{CC}, the authors introduced a certain structure of Hall
algebra involving the cluster category $\mathcal{C}(Q)$ by
associating the objects in $\mathcal{C}(Q)$ to some variables
 given by an explicit map $X_?$, called the
Caldero-Chapoton map. We denote by $X_{M}$ the variable (called
the generalized cluster variable) associated to an object $M$ in
$\mathcal{C}(Q).$ In the case that $M$ is a non-projective
$kQ$-module, the authors gave the the multiplication of $X_M$ and
$X_{\tau M}$ as follows:
\begin{equation}\label{CC-AR}
X_{\tau M}X_{M}=X_B+1
\end{equation}
where $B$ is the middle term of the almost split sequence
involving $M$ and $\tau M.$

If $Q$ is a simply laced Dynkin quiver, Caldero and Keller
\cite{CK2005} extended the above multiplication \eqref{CC-AR} to
the multiplication of any two variables associated to two
indecomposable objects in $\mathcal{C}(Q)$ as follows
$$
\chi_c(\mathbb{P}\mathrm{Ext}^{1}(M,N))X_MX_N=\sum_{Y}(\chi_c(\mathbb{P}\mathrm{Ext}^{1}(M,N)_Y)+\chi_c(\mathbb{P}\mathrm{Ext}^{1}(N,M)_Y))X_Y
$$
where $\chi_c$ is the Euler-Poincar\'e characteristic of \'etale
cohomology with proper support, $M,N\in \mathcal{C}(Q)$ and $Y$
runs through the isoclasses of $\mathcal{C}(Q).$ This is called
the cluster multiplication theorem for finite type.

The above cluster multiplication theorem was generalized to the
affine type in \cite{Hubery2005} by using Green's theorem and the
existence of Hall polynomials for affine quivers.  A cluster
multiplication theorem for indecomposable regular modules over the
path algebra of an affine quiver was proved in \cite{dupont}. In
\cite{Palu}, the author gave the cluster multiplication theorem in
the case when $\mathrm{dim}_{k}\mathrm{Ext}^{1}(M,N)=1$ and
introduced the cluster character for an arbitrary 2-Calabi-Yau
category with cluster-tilting objects.

The aim of this paper is to generalize the cluster multiplication
theorems for finite and affine types to arbitrary type. Note that
there is an alternative proof of the cluster multiplication
theorem for arbitrary type in \cite{XX2007} by applying the
projective version of Green's theorem under the $\bbc^{*}$-action.
Compared with \cite{XX2007}, the present proof has the following
differences. First, it is more direct and simpler. The present
proof is independent of the projective version of Green's formula
and only involves the properties of high order associativity
(analogous to the associativity of the multiplication of a derived
Hall algebra defined in \cite{Toen2005}, see Section
\ref{highorder} for details) and 2-Calabi-Yau (Auslander-Reiten
formula).  Second, it is more accessible. The present proof uses
Euler characteristics of algebraic varieties instead of quasi
Euler characteristics of orbit spaces of algebraic varieties under
the actions of algebraic groups in \cite{XX2007}. Third, it is
more promising to apply the approaches in the present proof to
hereditary categories which are not module categories of
hereditary algebras.

The interaction between cluster algebras and representation theory
of quiver naturally makes us  ask the question whether there are
cluster algebras associated to the cluster categories of the
categories of coherent sheaves over weighted projective lines or
elliptic curves. Also, it is meaningful to ask what is the
corresponding cluster multiplication theorem. The intuitive idea
is to extend the method in \cite{XX2007}. However, the proof in
\cite{XX2007} heavily depends on Green's theorem which holds just
for module categories of hereditary algebras. Also, the proof of
the projective version of Green's formula in \cite{XX2007} is
complicated. We need to look for a new approach not using Green's
theorem.

The high order associativity in the present proof is analogous to
the associativity of the multiplication in derived Hall algebras
under combinatorial context. Hence, it is  hopeful that the
property of high order associativity holds for categories of
coherent sheaves over weighted projective lines or elliptic curves
if we discuss the property under the suitable geometric context
(see Remark \ref{final}). As for the property of 2-Calabi-Yau, the
situation is the same.

This paper is organized as follows. In Section \ref{basic}, we
recall the general theory involving the computation of Euler
characteristics of algebraic varieties and the cluster category
needed in this paper. In order to use the proposition in Section
\ref{basic} to compute Euler characteristics, we need construct
some morphisms of varieties. Section \ref{morphism} is contributed
to this aim. In Section \ref{highassociativity}, we prove an
equation called the high order associativity. The cluster
multiplication theorem for arbitrary type is stated and proved in
the last section. As an application of the proof of the main
theorem, we induce the formula \eqref{CC-AR}. Finally, we
illustrate our theorem through an example which has been studied
in detail in \cite{CZ} and \cite{Zelevinsky}.

\section{preliminaries}\label{basic}
Let $Q=(Q_0,Q_1,s,t)$ be an acyclic quiver, where $Q_0$ and $Q_1$
are the sets of vertices and arrows, respectively, and $s,t:
Q_1\rightarrow Q_0$ are maps such that any arrow $\az$ starts at
$s(\az)$ and terminates at $t(\az).$ Let $\bbc Q$ be the path
algebra of $Q$ over $\bbc$. We denote by $\mathrm{mod} \bbc Q$ the
category of finite dimension $\bbc Q$-modules. A map $\textbf d :
Q_0\rightarrow \bbn$ such that $Q_0\setminus \textbf d^{-1}(0)$ is
finite is called a \emph{dimension vector} for $Q.$
\subsection{Euler characteristics and the pushforward
functor}\label{geometric context}
  For any dimension vector
 ${\textbf d}=(d_i)_{i\in Q_0},$ we consider the affine space over $\bbc$
$$\bbe_{\textbf d}=\bbe_{\textbf d}(Q)=\bigoplus_{\az\in Q_1}\hom_{\bbc}(\bbc^{d_{s(\az)}},\bbc^{d_{t(\az)}}).$$
Any element $x=(x_{\az})_{\az\in Q_1}$ in $\bbe_{\textbf d}$
defines a representation $M(x)=(\bbc^{{\textbf d}},x)$ where
$\bbc^{{\textbf d}}=\bigoplus_{i\in Q_0}\bbc^{d_i}$. Naturally we
can define the action of the algebraic group $G_{{\textbf
d}}(Q)=\prod_{i\in Q_0}GL(\bbc^{d_i})$ on $\bbe_{\textbf d}$ by
$g.x=(g_{t(\alpha)}x_{\alpha}g^{-1}_{s(\alpha)})_{\alpha\in Q_1}.$

Let $X$ be an algebraic variety over $\bbc$. A constructible
function $f: X\rightarrow \bbq$ satisfies that $f(X)$ is a finite
subset of $\bbq$ and $f^{-1}(c)$ is a constructible subset of $X$
for any $c\in \bbq$. Write $M(X)$ for the $\mathbb{Q}$-vector
space of constructible functions over $X.$  Now, suppose
$\mathcal{O}$ is a constructible subset of $\bbe_{\textbf d}.$ The
function $1_{\mathcal{O}}$ is called a characteristic function if
$1_{\mathcal{O}}(x)=1$, for any $x\in \mathcal{O}$ and $0$
otherwise. It is clear that $1_{\mathcal{O}}$ is the simplest
constructible function and any constructible function is a linear
combination of characteristic functions.  We say $\mathcal{O}$ is
$G_{{\textbf d}}$-invariant if $G_{{\textbf
d}}\cdot\mathcal{O}=\mathcal{O}.$ In this case, $1_{\mo}$ is
called $G_{{\textbf d}}$-invariant.

In the following, the constructible sets and functions will always
be assumed $G_{{\textbf d}}$-invariant unless particularly
mentioned.

Let $\chi$ denote the Euler characteristic in compactly-supported
cohomology. Let $X$ be an algebraic variety and $\mo$ a
constructible subset of $X$ which is the disjoint union of
finitely many locally closed subsets $X_i$ for $i=1,\cdots,m.$
Define $\chi(\mo)=\sum_{i=1}^m\chi(X_i).$ We note that it is
well-defined. We have the following properties of $\chi$.
\begin{Prop}[\cite{Riedtmann} and \cite{Joyce}]\label{Euler} Let $X,Y$ be algebraic varieties over $\mathbb{C}.$
Then
\begin{enumerate}
    \item  If the algebraic variety $X$ is the disjoint union of
finitely many constructible sets $X_1,\cdots,X_r$, then
$$\chi(X)=\sum_{i=1}^{r}{\chi(X_i)}.$$
    \item  If $\varphi:X\longrightarrow Y$ is a morphism such that all fibers have the same Euler
characteristic $\chi$, then $\chi(X)=\chi\cdot \chi(Y).$
    \item $\chi(\bbc^n)=1$ and $\chi(\mathbb{P}^n)=n+1$ for all $n\geq
    0.$
\end{enumerate}
\end{Prop}
 We recall the {\it pushforward}
functor from the category of algebraic varieties over $\mathbb{C}$
to the category of $\mathbb{Q}$-vector spaces (see
\cite{Macpherson} and \cite{Joyce}). Let $\phi: X\rightarrow Y$ be
a morphism of varieties. For $f \in M(X)$ and $y\in Y,$ define
$$
\phi_{*}(f)(y)=\sum_{c\neq 0}c\chi(f^{-1}(c)\cap \phi^{-1}(y)).
$$
\begin{theorem}[\cite{Dimca},\cite{Joyce}]\label{Joyce}
Let $X,Y$ and $Z$ be algebraic varieties over $\mathbb{C},$ $\phi:
X\rightarrow Y$ and $\psi: Y\rightarrow Z$ be morphisms of
varieties, and $f\in M(X).$ Then $\phi_{*}(f)$ is constructible,
$\phi_{*}: M(X)\rightarrow M(Y)$ is a $\mathbb{Q}$-linear map and
$(\psi\phi)_{*}=\psi_{*}\phi_{*}$ is a $\mathbb{Q}$-linear map
from $M(X)$ to $M(Z).$
\end{theorem}

Given a $\bbc Q$-module $M$ and any dimension vector $\textbf e
\in \bbn^{Q_0}$, we denote by $\mathrm{Gr}_{\textbf e}(M)$ the set
of submodules $M_1 \subset M$ such that $\textbf{dim} M_1 =
\textbf e$. It is a closed subvariety of the product of
Grassmannians of subspaces $\prod_{i\in Q_0}
\mathrm{Gr}_{e_i}(\bbc^{d_i})$. Here, $\textbf{dim} M=\textbf d$.
Set
$$\mathrm{Gr}_{\textbf e}(\bbe_{\textbf d})=\{(M,M_1)\mid M\in
\bbe_{\textbf d},M_1\in \mathrm{Gr}_{\textbf e}(M)\}.$$
\begin{Prop}\label{finite}
Let $\textbf d$ and $\textbf e$ be two dimension vectors. Then the
function $f: \bbe_{\textbf d}\rightarrow \bbq$ sending $M$ to
$\chi(\mathrm{Gr}_{\textbf e}(M))$ is an $G_{\textbf d}$-invariant
constructible function.
\end{Prop}
\begin{proof} Consider the natural projection
$\phi: \mathrm{Gr}_{\textbf e}(\bbe_{\textbf d})\rightarrow
\bbe_{\textbf d}.$ It is clear that $\phi$ is surjective. By
Theorem \ref{Joyce}, we know that
$\phi_{*}(1_{\mathrm{Gr}_{\textbf e}(\bbe_{\textbf d})})=f$ is
constructible.
\end{proof}
For fixed $\textbf d,$  we can make finitely many choices of
$\textbf e$ such that $\mathrm{Gr}_{\textbf e}(\bbe_{\textbf d})$
is nonempty. For $M\in \bbe_{\textbf d},$ we define \cite[Section
1.2]{GLS2006}$$\str{M}:=\{M'\in \bbe_{\textbf d}\mid
\chi(\mathrm{Gr}_{\textbf e}(M'))=\chi(\mathrm{Gr}_{\textbf e}(M))
\mbox{ for any }\textbf e\}$$ Proposition \ref{finite} has the
following corollary.
\begin{Cor}\label{partition}
There exists a finite finite subset $S(\textbf d)$ of $\bbe_{\textbf
d}$ such that $$\bbe_{\textbf d}=\bigsqcup_{M\in S(\textbf
d)}\str{M}.$$
\end{Cor}

\bigskip

\subsection{The cluster category}
Given an acyclic quiver $Q$ and  $i \in Q_0$, we denote by $S_i$
the simple $\bbc Q$-module associated to $i$, by $P_i$ its
projective cover and by $I_i$ its injective hull. Given a $\bbc
Q$-module $M$, we denote by $\textbf{dim} M$ its dimension vector.
For any $i \in Q_0$, we will always denote by $s_i$ the $i$-th
vector of the canonical basis of $\bbz^{Q_0}$. In particular for
any $i \in Q_0$, we have $\textbf{dim} S_i=s_i$. We denote by
$\left <-,-\right >$ the Euler form on $\bbc Q$-mod given by
 $$\left <M, N\right >:=\left <\textbf{dim} M, \textbf{dim} N\right >=\mathrm{dim}_{\bbc} \hom_{\bbc Q}(M,N) - \mathrm{dim}_{\bbc} \ext^1_{\bbc Q}(M,N)$$
 for any $\bbc Q$-modules $M$ and $N$. In the following, for any additive category $\mathcal{F}$, we denote by $\mathrm{ind}(\mathcal{F})$ the subcategory of $\mathcal{F}$ formed
by a system of representatives of the isomorphism classes of
indecomposable objects in $\mathcal{F}.$

Let $\md^b(Q)$ be the bounded derived category of $\mathrm{mod}
\bbc Q$ with the shift functor $[1]$ and the AR-translation
$\tau$. The cluster category associated to $Q$ is the orbit
category $\mathcal{C}=\mathcal{C}(Q):=\md^{b}(Q)/F$ with
$F=[1]\tau^{-1}.$ It is proved in \cite{Keller} that $\mc$ is a
triangulated category with the canonical triangle functor
$\md^b(Q)\rightarrow \mc.$ As in \cite{BMRRT} and \cite{CK2005},
the category $\mathcal{C}$ is 2-Calabi-Yau, i.e., there is an
almost canonical non degenerate bifunctorial pairing
$$\phi: \ext^1_{\mc}(M,N)\times \ext^{1}_{\mc}(N,M)\rightarrow
\bbc.
$$
Here, the property of 2-Calabi-Yau is induced by Auslander-Reiten
formula $$\ext^1_{\bbc Q}(X,Y)\cong D\hom_{\bbc Q}(Y,\tau X)$$ for
$X,Y\in \mathrm{mod}\bbc Q.$ We can identify $\bbc Q$-modules with
their images in $\mc(Q)$ by considering the embedding of
$\mathrm{mod}\bbc Q$ into $\mc(Q).$ Each object $M$ in $\mc(Q)$
can be uniquely decomposed into the form: $M=M_0\oplus
P_M[1]=M_0\oplus\tau P_M$, where $M_0\in\mod \bbc Q$ and $P_M$ is
projective in $\mod \bbc Q.$

The Caldero-Chapton map of an acyclic quiver $Q$ is the map
$$X_?^Q: \mathrm{obj}(\mc(Q))\ra\bbq(x_1,\cdots,x_n)$$ defined in \cite{CC} by
the following rules:
            \begin{enumerate}
                \item if $M$ is an indecomposable $\bbc Q$-module, then
                    $$
                        X_M^Q = \sum_{\textbf e} \chi(\mathrm{Gr}_{\textbf e}(M)) \prod_{i \in Q_0} x_i^{-\left<\textbf e, s_i\right>-\left <s_i, \textbf{dim} M - \textbf e\right
                        >};
                    $$
                \item if $M=P_i[1]$ is the shift of the projective module associated to $i \in Q_0$, then $$X_M^Q=x_i;$$
                \item for any two objects $M,N$ of $\mathcal C_Q$,
                we have
                    $$X_{M \oplus N}^Q=X_M^QX_N^Q.$$
            \end{enumerate}
Here, for $v=(v_1,\cdots,v_n)\in \bbz^n$, we set
$$
x^v =x_1^{v_1}\cdots x_n^{v_n}.
$$

Without risk of confusion, we can write $X_?$ instead of $X^Q_?$.
Let $R=(r_{ij})$ be a matrix of size $|Q_0|\times |Q_0|$
satisfying
$$r_{ij}=\dim_{\bbc}\mathrm{Ext}^1(S_i,S_j)$$ for any $i, j\in Q_0$.
We need the following lemma \cite [Lemma 1]{Hubery2005} to
reformulate the Caldero-Chapoton map.
\begin{lemma}\label{Huberylemma1}
For any $\bbc Q$-module $M$ without projective summands, we have
$$
(\textbf{dim}M)R+(\textbf{dim}\tau
M)R^{tr}=\textbf{dim}M+\textbf{dim}\tau M
$$
For a projective $\bbc Q$-module $P$ and an injective module $I$, we
have
$$
(\textbf{dim}P)R=\textbf{dim}radP, \quad
(\textbf{dim}I)R^{tr}=\textbf{dim}socI
$$
\end{lemma}
Following this lemma, we reformulate the above map by the
following rules:

\begin{enumerate}
\item
$$
X_{\tau P}=X_{P[1]}=x^{\textbf{dim} {P/rad P}},
X_{\tau^{-1}I}=X_{I[-1]}=x^{\textbf{dim} socI}$$ for any
projective $\bbc Q$-module $P$ and any injective $\bbc Q$-module
$I$; \item
$$ X_{M}=\sum_{\textbf e}\chi(\mathrm{Gr}_{\textbf e}(M))x^{\textbf{e} R+(\textbf{dim} M-\textbf e)R^{tr}-
\textbf{dim} M }
$$
where $M$ is a $\bbc Q$-module and $R^{tr}$ is the transpose of
the matrix $R$.
\end{enumerate}

The following proposition shows that the above reformulation
induces the third rule in the definition of the Caldero-Chapoton
map. It is actually the degeneration form of Green's formula in
\cite{DXX}.
\begin{Prop}
For any $M, N\in \mathrm{mod}\bbc Q,$ we have $X_{M\oplus
N}=X_{M}X_{N}.$
\end{Prop}
\begin{proof}
It is enough to prove
$$\chi(\mathrm{Gr}_{\textbf e}(M\oplus
N))=\sum_{\textbf e_1+\textbf e_2=\textbf
e}\chi(\mathrm{Gr}_{\textbf e_1}(M))\cdot \chi(\mathrm{Gr}_{\textbf
e_2}(N)).
$$ We define a morphism of varieties:
$$
\mathrm{Gr}_{\textbf e}(M\oplus N)\rightarrow \bigsqcup_{\textbf
e_1+\textbf e_2=\textbf e}\mathrm{Gr}_{\textbf e_1}(M)\times
\mathrm{Gr}_{\textbf e_2}(N).
$$
Any submodule $L$ of $M\oplus N$ induces uniquely the submodules
$M_1$  and $N_1$ of $M$ and $N$, respectively. It is a surjective
morphism. The fibre for any $(M_1,N_1)\in \mathrm{Gr}_{\textbf
e_1}(M)\times \mathrm{Gr}_{\textbf e_2}(N)$ is isomorphic to an
affine space by \cite[Lemma 7]{Hubery2005}. Then, Proposition
\ref{Euler} induces the identity of the proposition.
\end{proof}
For any $M\in \mod\bbc Q,$ we say that $P_0$ is the maximal
projective direct summand of $M$ if $M\simeq M'\oplus P_0$ as the
$\bbc Q$-modules and $M'$ does not contain projective direct
summands.

Let $\tilde{A}(Q)$ be the subalgebra of $\bbq(x_1,\cdots,x_n)$
generated by
$$\{X_{M}, X_{\tau P}|M, P \in\ind(\mathrm{mod} \bbc Q) \mbox{ and } P \ \mbox{is
projective}\ \}$$ and $A(Q)$ the subalgebra of $\tilde{A}(Q)$
generated by $$\{X_{M}, X_{\tau P}|M, P \in\ind(\mathrm{mod}\bbc
Q),  P \ \mbox{is projective}  \mbox{ and } \ext^1(M,M)=0\ \}.$$
The algebra $A(Q)$ is called the cluster algebra associated to
$Q$. If $Q$ is of finite type, then $\tilde{A}(Q)$ is just the
cluster algebra $A(Q)$ as shown in \cite{CK2006}. We note that the
relation between $\tilde{A}(Q)$ and $A(Q)$ is different from the
relation between the Ringel-Hall algebra and the composition
algebra for $Q$ (see Section \ref{example}).

\section{Morphisms of varieties induced by kernel, cokernel and AR-translation}\label{morphism}
The cluster multiplication theorem for arbitrary type in Section 4
will be expressed under the context of $\bbc Q$-modules. In the
sequel, we will only consider the restriction of AR-translation
$\tau$ to $\mathrm{mod} \bbc Q,$ rather than the cluster category.
We use the same notation without risk of confusion. Hence, $\tau
P=\tau^{-1}I=0$ for any projective module $P$ and injective module
$I$. In this section, we define morphisms induced by kernel,
cokernel and AR-translation $\tau$. These morphisms guarantee that
we can use Proposition \ref{Euler}.

\subsection{Morphisms induced by kernel and cokernel}\label{morphismkernel}
Let $(\bbc^{\textbf d}, x)$ and $(\bbc^{\textbf d'}, x')$ be two
$\bbc Q$-modules. In this subsection, for any $f\in
\mathrm{Hom}_{\bbc Q}((\bbc^{\textbf d}, x), (\bbc^{\textbf d'},
x'))$, we will describe $\mathrm{ker}f$, $\mathrm{Im}f$ and
$\mathrm{\mathrm{coker}}f$ under geometric context. The main
barrier is that the underlying spaces for $\mathrm{ker}f$,
$\mathrm{Im}f$ and $\mathrm{\mathrm{coker}}f$ are not of the form
$\bbc^{\textbf e}$ for some dimension vector $\textbf e$. First,
we deal with the case of vector spaces.

Let $\bbc^d$ and $\bbc^{d'}$ be two vector spaces of dimension $d$
and $d'$, respectively. Let $M_{d'\times d}$ be the set of all
matrices of size $d'\times d.$ Then $M_{d'\times d}=\hom(\bbc^d,
\bbc^{d'})$ and $M_{d'\times d}=\bigsqcup_{r}M_{d'\times d}(r)$
where $M_{d'\times d}(r)$ consists of all matrices of rank $r$.
For any $A=(a_{ij})\in M_{d'\times d}(r),$ let us denote the
$r\times r$ submatrix of $A$ formed by the rows $1\leq i_1<\cdots
<i_r\leq d'$ and the columns $1\leq j_1<\cdots <j_r\leq d$ by
$\bigtriangleup_{(i_1,\cdots, i_r; j_1,\cdots, j_r)}(A)$. For
every pair of multi-indices $I=\{i_1,\cdots, i_r\}\subseteq \{1,
\cdots, d'\}$ and $J=\{j_1,\cdots, j_r\}\subseteq \{1, \cdots,
d\},$ we define $M_{d'\times d}(r,I,J)$ to be the subset of
$M_{d'\times d}(r)$ consisting of the matrices $A$ which satisfy
$A\notin M_{d'\times d}(I', J')$ for any $I'<I$ or $I'=I, J'<J$
and $\mathrm{det}\ \bigtriangleup_{(i_1,\cdots, i_r; j_1,\cdots,
j_r)}(A)\neq 0$. Here $I'<I$ is the common lexical order. We have
a finite stratification of $M_{d'\times d}(r)$, i.e.,
$$
M_{d'\times d}(r)=\bigsqcup_{(I,J)}M_{d'\times d}(r,I,J).
$$
In particular, if $d<d'$ and $r=d$, this gives a finite
stratification of the Grassmannian $\mathrm{Gr}_{d}(\bbc^{d'})$
consisting of all $d$-dimensional subspaces of $\bbc^{d'}.$
Indeed, for any $I=\{i_1,\cdots, i_d\}\subset \{1, \cdots, d'\}$
and $J=\{1,\cdots, d\}$, let $M^g_{d'\times d}(I)$ be the subset
of $M_{d'\times d}(d, I,J)$ consisting of the matrices $A$
satisfying that $\bigtriangleup_{(I;J)}(A)$ are identity matrices.
Then there is a finite stratification
$$
\mathrm{Gr}_{d}(\bbc^{d'})=\bigsqcup_{I}M^g_{d'\times d}(I).
$$
For any $A\in M_{d'\times d}(d, I, J)$, we substitute the identity
matrix for the submatrix $\bigtriangleup_{(I;J)}(A)$ and then $A$
corresponds to a unqiue matrix $A'\in M^g_{d'\times d}(I)$.

For every pair of multi-indices $I=\{i_1,\cdots, i_r\}$ and
$J=\{j_1,\cdots, j_r\},$ we will define the following morphism of
varieties:
 $$\Upsilon_{(r, I, J)}^{1}: M_{d'\times d}(r, I,
J)\rightarrow M_{d\times (d-r)}(d-r),$$
$$\hspace{-0.8cm}\Upsilon_{(r, I,
J)}^{2}: M_{d'\times d}(r, I, J)\rightarrow
\mathrm{Gr}_{d-r}(\bbc^{d}),$$
$$\hspace{0.3cm}\Omega_{(r, I, J)}^{1}: M_{d'\times
d}(r, I, J)\rightarrow M_{(d'-r)\times d'}(d'-r), $$ and
$$\hspace{-1.2cm}\Omega_{(r, I, J)}^2: M_{d'\times
d}(r, I, J)\rightarrow \mathrm{Gr}_{r}(\bbc^{d'}).$$ Let
$P_{ij}(k)$ be the elementary matrix of size $k\times k$
transposing the $i$-th row and the $j$-th row. Set
$P_{I}(d')=P_{r,i_r}(d')\cdots P_{1, i_1}(d')$ and
$P_{J}(d)=P_{r,j_r}(d)\cdots P_{1, j_1}(d)$. Then we have
$$P_{I}(d')AP_J(d)\in M_{d'\times d}(r,(1,\cdots,r)(1,\cdots,r))$$ for
any matrix $A\in
M_{d'\times d}(r, I, J)$. The matrix $P_{I}(d')AP_J(d)$ has the form $\left(%
\begin{array}{cc}
  A_1 & A_2 \\
  A_3 & A_4 \\
\end{array}%
\right)$ with  an invertible $r\times r$ matrix $A_1$ and
$A_4=A_3A^{-1}_1A_2=A_2A^{-1}_1A_3.$ The matrix $P_J(d)\left(%
\begin{array}{c}
  -A_1^{-1}A_2 \\
  I_{d-r} \\
\end{array}%
\right)$ determines the solution space $\{x\in \bbc^{d}\mid
Ax=0\}$. The matrix $(-A_3A^{-1}_1, I_{d'-r})P_I(d')$ determines
the solution space $\{x\in \bbc^{d}\mid xA=0\}$.
We define $$\Upsilon^1_{(r, I, J)}(A)=P_J(d)\left(%
\begin{array}{c}
  -A_1^{-1}A_2 \\
  I_{d-r} \\
\end{array}%
\right), $$ $$\Omega^1_{(r, I, J)}(A)=(-A_3A^{-1}_1,
I_{d'-r})P_I(d').$$

Assume that $P_J(d)\left(%
\begin{array}{c}
  -A_1^{-1}A_2 \\
  I_{d-r} \\
\end{array}%
\right)\in M_{d\times (d-r)}(d-r, I', J')$ for some $I'\subset
\{1,\cdots, d\}$ and $J'=(1,\cdots, d-r)$.  Then we define
$\Upsilon^2_{(r, I, J)}(A)$ to be the
unique matrix in $M^g_{d\times (d-r)}(I')$ which the matrix $P_J(d)\left(%
\begin{array}{c}
  -A_1^{-1}A_2 \\
  I_{d-r} \\
\end{array}%
\right)$ corresponds to. Similarly, we define $\Omega^2_{(r, I,
J)}(A)$ to be the unique matrix in $M^g_{d'\times r}(I)$ which the
submatrix $\bigtriangleup_{(1,\cdots, d'; j_1,\cdots, j_r})(A)$ of
$A$ corresponds to. Hence, for any $A\in M_{d'\times d}(r, I, J)$,
we have a long exact sequence of $\bbc$-spaces
$$
\xymatrix{0\ar[r]&\bbc^{d-r}\ar[rr]^{\Upsilon^1_{(r, I,
J)}(A)}&&\bbc^{d}\ar[r]^{A}&\bbc^{d'}\ar[rr]^{\Omega^1_{(r, I,
J)}(A)}&&\bbc^{d'-r}\ar[r]&0}.
$$

  Now, we consider the $\bbc Q$-module. Let $(\bbc^{\textbf d}, x)$ and $(\bbc^{\textbf d'}, x')$ be two $\bbc Q$-modules with dimension vectors $\textbf d$ and $\textbf d'$, respectively.
For any morphism of $\bbc Q$-modules $f: (\bbc^{\textbf d},
x)\rightarrow (\bbc^{\textbf d'}, x')$, we have $f=(f_i)_{i\in Q_0}$
with $f_i\in M_{d'_i\times d_i}$ and
$f_{t(\alpha)}x_{\alpha}=x'_{\alpha}f_{s(\alpha)}$ for any $i\in
Q_0$ and $\alpha\in Q_1$. There is a finite stratification of
$M_{d'_i\times d_i}$ for any $i\in Q_0$ as follows:
$$M_{d'_i\times d_i}=\bigsqcup_{r_i, I_i, J_i}
M_{d'_i\times d_i}(r_i, I_i, J_i).$$ Then $\mathrm{Hom}_{\bbc
Q}((\bbc^{\textbf d},x), (\bbc^{\textbf d'},x'))$ is a closed
subset of
$$
\prod_{i\in Q_0}M_{d'_i\times d_i}=\bigsqcup_{((r_i, I_i,
J_i))_{i\in Q_0}}\prod_{i\in Q_0}M_{d'_i\times d_i}(r_i, I_i,
J_i).
$$
This induces a finite stratification of $\mathrm{Hom}_{\bbc
Q}((\bbc^{\textbf d},x), (\bbc^{\textbf d'},x'))$.

For any $i\in Q_0,$ we fix a pair of multi-indices $I_i=(k_{i1},
\cdots, k_{ir_i})$ and $J_i=(l_{i1}, \cdots, l_{ir_i}).$ We have
the following morphism of varieties:
 $$\Upsilon^1_{((r_i, I_i, J_i))_{i\in Q_0}}: =\prod_{i\in Q_0}\Upsilon^1_{(r_i, I_i, J_i)}: \prod_{i\in Q_0}M_{d'_i\times d_i}(r_i, I_i,
J_i)\rightarrow \prod_{i\in Q_0}M_{d_i\times
(d_i-r_i)}(d_i-r_i),$$
$$\Upsilon^2_{((r_i, I_i, J_i))_{i\in Q_0}}:=\prod_{i\in Q_0}\Upsilon^2_{(r_i, I_i, J_i)}: \prod_{i\in Q_0}M_{d'_i\times d_i}(r_i, I_i, J_i)\rightarrow
\prod_{i\in Q_0}\mathrm{Gr}_{d_i-r_i}(\bbc^{d_i}),$$
$$\Omega^1_{((r_i, I_i, J_i))_{i\in Q_0}}:=\prod_{i\in Q_0}\Omega^1_{(r_i, I_i, J_i)}: \prod_{i\in Q_0}M_{d'_i\times
d_i}(r_i, I_i, J_i)\rightarrow \prod_{i\in Q_0}M_{(d'_i-r_i)\times
d'_i}(d'_i-r_i),
$$ and
$$\Omega^2_{((r_i, I_i, J_i))_{i\in Q_0}}:=\prod_{i\in Q_0}\Omega^2_{(r_i, I_i, J_i)}:  \prod_{i\in Q_0}M_{d'_i\times
d_i}(r_i, I_i, J_i)\rightarrow \prod_{i\in
Q_0}\mathrm{Gr}_{r_i}(\bbc^{d'_i}).$$ Without loss of generality,
we can assume the above $\bbc Q$-module homomorphism $f\in
\prod_{i\in Q_0}M_{d'_i\times d_i}(r_i, I_i, J_i).$ Then we have
$\Upsilon^2_{((r_i, I_i, J_i))_{i\in Q_0}}(f)\in
\mathrm{Gr}_{\textbf d-\textbf r}((\bbc^{\textbf d}, x))$ with
$\textbf r=(r_i)_{i\in Q_0}.$ Indeed, since
$$f_{t(\alpha)}x_{\alpha}(\Upsilon^2_{(r_{s(\alpha)}, I_{s(\alpha)},
J_{s(\alpha)})}(f_{s(\alpha)}))=x'_{\alpha}f_{s(\alpha)}(\Upsilon^2_{(r_{s(\alpha)},
I_{s(\alpha)}, J_{s(\alpha)})}(f_{s(\alpha)}))=0,$$ then
$x_{\alpha}(\Upsilon^2_{(r_{s(\alpha)}, I_{s(\alpha)},
J_{s(\alpha)})}(f_{s(\alpha)}))\in \Upsilon^2_{(r_{t(\alpha)},
I_{t(\alpha)}, J_{t(\alpha)})}(f_{t(\alpha)}).$ In fact,
$$\Upsilon^2_{((r_i, I_i, J_i))_{i\in Q_0}}(f)=\mathrm{ker} f.$$ As
discussed in the case of matrix, for any $i\in Q_0$, we assume
$P_{I_i}(d'_i)f_iP_{J_i}(d_i)$ has the form $\left(%
\begin{array}{cc}
  A_1(f_i) & A_2(f_i) \\
  A_3(f_i) & A_4(f_i) \\
\end{array}%
\right)$ with  an invertible $r_i\times r_i$ matrix $A_1(f_i)$ and
$A_4(f_i)=A_3(f_i)(A_1(f_i))^{-1}A_2(f_i)=A_2(f_i)(A_1(f_i))^{-1}A_3(f_i).$
Then we have a $\bbc Q$-module $(\bbc^{\textbf d-\textbf r}, y)$
isomorphic to $\mathrm{ker} f$ such that
$$
y_{\alpha}=\left(
             \begin{array}{cc}
               0 & I_{d_{t(\alpha)}-r_{t(\alpha)}} \\
             \end{array}
           \right)P_{J_{t(\alpha)}}(d_{t(\alpha)})x_{\alpha}P_{J_{s(\alpha)}}(d_{s(\alpha)})\left(
                                                                \begin{array}{c}
                                                                  -(A_1(f_{s(\alpha)}))^{-1}A_2(f_{s(\alpha)}) \\
                                                                  I_{d_{s(\alpha)}-r_{s(\alpha)}}\\
                                                                \end{array}
                                                              \right).
$$
Now we can write down the following left exact sequence of $\bbc
Q$-modules
$$
\xymatrix{0\ar[r]& (\bbc^{\textbf d-\textbf r},
y)\ar[rrr]^-{\Upsilon^1_{((r_i, I_i, J_i))_{i\in
Q_0}}(f)}&&&(\bbc^{\textbf d}, x)\ar[r]^{f}&(\bbc^{\textbf d'},
x')}.
$$
By similar discussion, we obtain
$$
\Omega^2_{((r_i, I_i, J_i))_{i\in Q_0}}(f)=\mathrm{Im} f.
$$
and a $\bbc Q$-module $(\bbc^{\textbf d'-\textbf r}, y')$ isomorphic
to $\mathrm{coker} f$ satisfying the following right exact sequence
$$
\xymatrix{(\bbc^{\textbf d}, x)\ar[r]^{f}& (\bbc^{\textbf d'},
x')\ar[rrr]^-{\Omega^1_{((r_i, I_i, J_i))_{i\in
Q_0}}(f)}&&&(\bbc^{\textbf d'-\textbf r}, y')\ar[r]& 0}.
$$
Therefore, we have the following proposition.
\begin{Prop}\label{Kernelcokernel}
For any $\bbc Q$-modules  $(\bbc^{\textbf d},x)$ and
$(\bbc^{\textbf d'},x'),$ there are the following maps whose
restriction to stratifications are morphisms of varieties
$$
\Upsilon^1: \mathrm{Hom}_{\bbc Q}((\bbc^{\textbf d},x),
(\bbc^{\textbf d'},x'))\rightarrow \prod_{\textbf
e}\mathrm{Inj}(\bbe_{\textbf e}(Q), (\bbc^{\textbf d},x)),
$$
$$
\hspace{-1cm}\Upsilon^2: \mathrm{Hom}_{\bbc Q}((\bbc^{\textbf
d},x), (\bbc^{\textbf d'},x'))\rightarrow \prod_{\textbf
e}\mathrm{Gr}_{\textbf e}((\bbc^{\textbf d},x)),
$$
$$
\hspace{0.1cm}\Omega^1: \mathrm{Hom}_{\bbc Q}((\bbc^{\textbf
d},x), (\bbc^{\textbf d'},x'))\rightarrow \prod_{\textbf
f}\mathrm{Surj}( (\bbc^{\textbf d'},x'), \bbe_{\textbf f}(Q)),
$$
$$
\hspace{-0.8cm}\Omega^2: \mathrm{Hom}_{\bbc Q}((\bbc^{\textbf
d},x), (\bbc^{\textbf d'},x'))\rightarrow \prod_{\textbf
e'}\mathrm{Gr}_{\textbf e'}((\bbc^{\textbf d'},x')),
$$
where $\mathrm{Inj}(\bbe_{\textbf e}(Q), (\bbc^{\textbf d},x))$ is
the set$$\{((\bbc^{\textbf e}, y), g)\mid g: (\bbc^{\textbf e},
y)\rightarrow (\bbc^{\textbf d}, x) \mbox{ is an injective }\bbc
Q\mbox{-homomorphism }\}$$ and $\mathrm{Surj}( (\bbc^{\textbf
d'},x'), \bbe_{\textbf f}(Q))$ is the set $$\{((\bbc^{\textbf f},
y'), g')\mid g': (\bbc^{\textbf d'}, x')\rightarrow (\bbc^{\textbf
f}, y') \mbox{ is an surjective }\bbc Q\mbox{-homomorphism }\}$$
satisfying that for any $f\in \mathrm{Hom}_{\bbc Q}((\bbc^{\textbf
d},x), (\bbc^{\textbf d'},x')),$ there exists a long exact
sequence of $\bbc Q$-modules
$$
\xymatrix{0\ar[r]&(\bbc^{\textbf e},
y)\ar[r]^-{\Upsilon^1(f)}&(\bbc^{\textbf
d},x)\ar[r]^{f}&(\bbc^{\textbf
d'},x')\ar[r]^-{\Omega^1(f)}&(\bbc^{\textbf f}, y')\ar[r]&0}.
$$
\end{Prop}
By the proof of Proposition \ref{Kernelcokernel}, we also have the
following corollary.
\begin{Cor}\label{grassmanian-submodule}
For any $\bbc Q$-modules  $(\bbc^{\textbf d},x)$ and dimension
vector $\textbf e$, there are the following maps whose restriction
to stratifications are morphisms of varieties
$$
\Upsilon_0: \mathrm{Gr}_{\textbf e}((\bbc^{\textbf d},
x))\rightarrow \bbe_{\textbf e}(Q) \quad \mbox{ and } \quad
\Omega_0: \mathrm{Gr}_{\textbf e}((\bbc^{\textbf d},
x))\rightarrow \bbe_{\textbf d-\textbf e}(Q)
$$
such that for any $M\in \mathrm{Gr}_{\textbf e}((\bbc^{\textbf d},
x))$, as the $\bbc Q$-modules, $\Upsilon_0(M)\cong M$ and
$\Omega_0(M)\cong (\bbc^{\textbf d},x)/M$.
\end{Cor}
\subsection{Morphisms induced by AR-translation $\tau$}
In this subsection, we will describe the Auslander-Reiten translation $\tau$ under geometric context.
Let $\Phi^{+}, \Phi^{-}$ be the Coxeter functors introduced by
Bernsetein, Gelfand and Ponomarev. We denote by $T$ the
endofunctor of $\mathrm{mod} \bbc Q$ sending $(\bbc^{\textbf d},
x)$ to $(\bbc^{\textbf d}, -x)$. Then the functor $T\Phi^{+}$ on
$\mathrm{mod} \bbc Q$ is just the AR-translation $\tau$ on
$\mathrm{mod} \bbc Q$.

Given any $\bbc Q$-module $(\bbc^{\textbf d}, x)\in \bbe_{\textbf
d}$, the representation $$\Phi^{+}(\bbc^{\textbf d},
x):=(\bbc^{\textbf e}, y)$$ can be constructed inductively as
described in \cite{Ringel1998}. Let us recall it. Since $Q$ is
acyclic, one can define a partial order on $Q_0$ such that for any
arrow $\beta,$ $s(\beta)>t(\beta).$ Assume that the dimension
$e_j$ with $j<i$, the linear maps $h_{\beta}:
\bbc^{e_{t(\beta)}}\rightarrow \bbc^{d_{s(\beta)}}$ for all
$\beta\in Q_1$ with $s(\beta)\leq i$ and the maps $y_{\alpha}:
\bbc^{e_{s(\alpha)}}\rightarrow \bbc^{e_{t(\alpha)}}$ for all
$\alpha\in Q_1$ with $s(\alpha)<i$ are defined. Then we have the
sequence
\begin{equation}\label{BGP}
\xymatrix{0\ar[r]& \bbc^{e_i}\ar[rr]^-{\Upsilon^1((x_{\alpha},
h_{\beta})_{\alpha,\beta})}&& \bigoplus_{t(\alpha)=i}
\bbc^{d_{s(\alpha)}}\oplus\bigoplus_{s(\beta)=i}\bbc^{e_{t(\beta)}}\ar[rr]^-{(x_{\alpha},
h_{\beta})_{\alpha,\beta}}&& \bbc^{d_{i}}}
\end{equation}
where $e_i$ is the dimension of the kernel of the map
$(x_{\alpha}, h_{\beta})_{\alpha,\beta}$ and $\Upsilon^1$ was
defined as in Section 2.1. Now we define the map $y_{\beta}:
\bbc^{e_{s(\beta)}}\rightarrow \bbc^{e_{t(\beta)}}$ for any
$\beta$ with $s(\beta)=i$ to be the composition
$$
\xymatrix{\bbc^{e_i}\ar[rr]^-{\Upsilon^1((x_{\alpha},
h_{\beta})_{\alpha,\beta})}&& \bigoplus_{t(\alpha)=i}
\bbc^{d_{s(\alpha)}}\oplus\bigoplus_{s(\beta)=i}\bbc^{e_{t(\beta)}}\ar[r]&
\bbc^{e_{t(\beta)}}}
$$
where the second map is the projection. Inductively we obtain the
representation $\tau(\bbc^{\textbf d}, x)=T\Phi^{+}(\bbc^{\textbf
d}, x)=(\bbc^{\textbf e}, -y)$. We write $\tau(\textbf d)=\textbf
e.$ The geometric construction of $\Phi^{-}$ (also $\tau^{-}$) is
similar. For any $\bbc Q$-module $(\bbc^{\textbf d}, x)\in
\bbe_{\textbf d},$ let $(\bbc^{\textbf d^{+}}, x^{+})$ and
$P_0(\textbf d, x)$ be its maximal non-projective summand and the
maximal projective summand, respectively, i.e., $(\bbc^{\textbf
d}, x)=(\bbc^{\textbf d^{+}}, x^{+})\oplus P_0(\textbf d, x)$
satisfying that $(\bbc^{\textbf d^{+}}, x^{+})$ contains no
projective summands and $P_0$ is projective. In fact, we have the
isomorphism of $\bbc Q$-modules
$$
\tau^{-1}\tau (\bbc^{\textbf d}, x)\cong (\bbc^{\textbf d^{+}},
x^{+}).
$$

We can explicitly write down the submodule $(V, x):=\tau^{-1}\tau
(\bbc^{\textbf d}, x)$ of $(\bbc^{\textbf d}, x)$. The space $V_i$
is just the image of $(x_\alpha, h_{\beta})_{\alpha, \beta}$ for
any $i\in Q_0$ in the sequence \eqref{BGP}. Indeed, Dually, we
denote by $(\bbc^{\textbf d^{-}}, x^{-})$ and $I_0(\textbf d,x)$
the maximal non-injective summand and the maximal injective
summand of $(\bbc^{\textbf d}, x),$ respectively. Then
$$
\tau\tau^{-1} (\bbc^{\textbf d}, x)\cong (\bbc^{\textbf d^{-}},
x^{-}).
$$
The above construction and its duality induces the following two
propositions.

\begin{Prop}\label{AR1}
For any dimension vector $\textbf d$,  there exists a morphism of
varieties
$$
\phi^{+}: \bbe_{\textbf d}\rightarrow \prod_{\tau(\textbf
d^{+})}\bbe_{ \tau(\textbf d^{+})}
$$
such that $\phi^{+}((\bbc^{\textbf d}, x))=\tau(\bbc^{\textbf d},
x).$
\end{Prop}
Dually, we have
\begin{Prop}\label{AR2}
For any dimension vector $\textbf d$,  there exists a morphism of
varieties
$$
\phi^{-}: \bbe_{\textbf d}\rightarrow \prod_{\tau^{-1}(\textbf
d^{-})}\bbe_{ \tau^{-1}(\textbf d^{-})}
$$
such that $\phi^{-}((\bbc^{\textbf d}, x))=\tau^{-1}(\bbc^{\textbf
d}, x).$
\end{Prop}

Let $f=(f_i)_{i\in Q_0}: (\bbc^{\textbf d}, x)\rightarrow
(\bbc^{\textbf d'}, x')$ be any morphism of $\bbc Q$-modules. Let
$\Phi^{+}((\bbc^{\textbf d}, x))=(\bbc^{\textbf e}, y)$ and
$\Phi^{+}((\bbc^{\textbf d'}, x'))=(\bbc^{\textbf e'}, y').$ Then
we can inductively construct the maps $g=(g_i)_{i\in Q_0}:
(\bbc^{\textbf e}, y)\rightarrow (\bbc^{\textbf e'}, y')$ by the
following commutative diagram:
$$
\xymatrix{0\ar[r]& \bbc^{e_i}\ar[d]^{g_i}\ar[rr]^-{(h_{\alpha},
y_{\beta})_{\alpha,\beta}}&& \bigoplus_{t(\alpha)=i}
\bbc^{d_{s(\alpha)}}\oplus\bigoplus_{s(\beta)=i}\bbc^{e_{t(\beta)}}\ar[d]^{\oplus
f_{s(\alpha)}\oplus g_{t(\beta)}}\ar[rr]^-{(x_{\alpha},
h_{\beta})_{\alpha,\beta}}&& \bbc^{d_{i}}\ar[d]^{f_i}\\
0\ar[r]& \bbc^{e'_i}\ar[rr]^-{(h'_{\alpha},
y'_{\beta})_{\alpha,\beta}}&& \bigoplus_{t(\alpha)=i}
\bbc^{d'_{s(\alpha)}}\oplus\bigoplus_{s(\beta)=i}\bbc^{e'_{t(\beta)}}\ar[rr]^-{(x_{\alpha},
h_{\beta})_{\alpha,\beta}}&& \bbc^{d'_{i}}}
$$
The commutativity of the diagram guarantees that $g=(g_i)$ is the
morphism of $\bbc Q$-modules. Hence, by using Proposition
\ref{Kernelcokernel}, Corollary \ref{grassmanian-submodule},
Proposition \ref{AR1} and \ref{AR2}, we have the following result.
\begin{Prop}\label{AR-grassmanian}
For any dimension vectors $\textbf d$,  there exists a morphism of
varieties
$$
\xymatrix{\mathrm{Gr}_{\textbf e}((\bbc^{\textbf
d},x))\ar[r]^-{g_{\tau}}& \mathrm{Gr}_{\tau(\textbf
e)}((\bbc^{\tau(\textbf d)},\tau x))}
$$

\end{Prop}

\section{High order associativity }\label{highassociativity}
In \cite{Toen2005}, To\"en associated an associative algebra
(called the derived Hall algebra) to a dg category over a finite
field $k$. In particular, we can define the derived Hall algebra
$\md\mh(Q)$ for the derived category $\md^b(\mathrm{mod} kQ)$. Let
$u_X$ denotes the isomorphism class of $X\in \md^b(\mathrm{mod}
kQ).$ The algebra $\md\mh(Q)$ is an associative algebra generated
by $u_X$ for any $X\in \md^b(\mathrm{mod} kQ)$. The associative
multiplication contains a non-trivial case as follows. For any
$L_1, L_2\in \mod kQ,$ we have
$$
u_{L_2}*u_{L_1[1]}=\sum_{[\theta], \theta: L_1\rightarrow
L_2}g_{L_1[1]L_2}^{K[1]\oplus C}u_{K[1]\oplus C}
$$
where $g_{L_1[1]L_2}^{K[1]\oplus C}\in \bbq$ is called the derived
Hall number and $K=\mathrm{ker}\theta, C=\mathrm{coker}\theta$ and
$[\theta]$ is the equivalence class of $\theta.$ Here $\theta_1$ is
equivalent to $\theta_2$ if there exist automorphisms $a_{L_1}$ and
$a_{L_2}$ such that $\theta_1 a_{L_1}=a_{L_2}\theta_2.$ The above
equation implies the following exact sequence
$$
\xymatrix{0\ar[r]&K\ar[r]&L_1\ar[r]^{\theta}& L_2\ar[r]& C\ar[r]&0}.
$$

By the associativity of the multiplication of $\md\mh(Q)$, we have
$$
u_{L_2}*(u_{K_1[1]}*u_{L_1[1]})=(u_{K_1[1]}*u_{L_2})*u_{L_1[1]}=u_{K_1[1]}*(u_{L_2}*u_{L_1[1]})
$$
for any $kQ$-modules $K_1, L_1, L_2$ and
$$
(u_{L_2}*u_{L_1[1]})*u_{C_2}=u_{L_2}*(u_{L_1[1]}*u_{C_2})
$$
for any $kQ$-modules $C_2, L_1, L_2$. These two equations can be
illustrated by the following commutative diagrams.
$$
\xymatrix{&K_1\ar[d]\ar@{=}[r]& K_1\ar[d]&&&\\
0\ar[r]&K\ar[d]\ar[r]&L\ar[d]\ar[r]&L_2\ar@{=}[d]\ar[r]& C
\ar@{=}[d]\ar[r]&
0\\
0\ar[r]&K_2\ar[r]&L_1\ar[r]&L_2\ar[r]&C\ar[r]&0}
$$
and
$$
\xymatrix{0\ar[r]&K\ar@{=}[d]\ar[r]&L_1\ar@{=}[d]\ar[r]&L_2\ar[d]\ar[r]&
C_1 \ar[d]\ar[r]&
0\\
0\ar[r]&K\ar[r]&L_1\ar[r]&L\ar[r]\ar[d]&C\ar[r]\ar[d]&0\\
&&&C_2\ar@{=}[r]&C_2&}
$$
We note that the long exact sequences in the above diagrams can be
decomposed into short exact sequences so that the above two
equations can be induced by using the associativity of the
multiplication of the Ringel-Hall algebra for two times. In this
section, we will prove an identity analogous to the above
associativity of the multiplication of the derived Hall algebra in
the context of Euler characteristic.  The identity is called the
high order associativity .

\subsection{The description of $\ext^1_{\bbc Q}(M,N)$}\label{extension}
Let $M\in \bbe_{\textbf d_1}, N\in \bbe_{\textbf d_2}$  and
$m(M,N)$ be the vector space over $\bbc$ of all tuples
$m=(m(\alpha))_{\alpha\in Q_1}$ such that linear maps
$m(\alpha)\in \mathrm{Hom}_{\bbc}(M_{s(\alpha)},N_{t(\alpha)})$
for all $\alpha \in Q_1$. We define a linear map $\pi:
m(M,N)\rightarrow \mathrm{Ext}^1(M,N)$ by mapping $m\in m(M,N)$ to
a short exact sequence
$$
\xymatrix{\varepsilon:\quad 0\ar[r]& N\ar[rr]^{\left(%
\begin{array}{c}
  1 \\
  0 \\
\end{array}%
\right)}&&L(m)\ar[rr]^{\left(%
\begin{array}{cc}
  0 & 1 \\
\end{array}%
\right)}&&M\ar[r]&0}
$$
where as a vector space, $L(m)$ is the direct sum of $Y$ and $X$ .
For any $\alpha\in Q_1,$
$$
L(m)_{\alpha}=\left(%
\begin{array}{cc}
  N_{\alpha} & m(\alpha) \\
  0 & M_{\alpha} \\
\end{array}%
\right).
$$
Let us fix a vector space decomposition
$m(X,Y)=\mathrm{ker}\pi\oplus E(M,N),$ then we can identify
$\mathrm{Ext}^{1}(M,N)$ with $E(M,N)$. There is a natural
$\bbc^{*}$-action on $E(M,N)$ by defining $t.m=(tm(\alpha))$ for
any $t\in \bbc^{*}.$ This action induces the action of $\bbc^{*}$
on $\mathrm{Ext}^{1}(M,N).$ Considering the isomorphism of $\bbc
Q$-modules between $L(m)$ and $L(t.m),$ we know that
$t.\varepsilon$ is the following short exact sequence:
$$
\xymatrix{0\ar[r]& N\ar[rr]^{\left(%
\begin{array}{c}
  t \\
  0 \\
\end{array}%
\right)}&&L(m)\ar[rr]^{\left(%
\begin{array}{cc}
  0 & 1 \\
\end{array}%
\right)}&&M\ar[r]&0}
$$
for any $t\in \bbc^{*}.$ Let $\mathrm{Ext}^{1}(M,N)_{L}$ be the
subset of $\mathrm{Ext}^{1}(M,N)$ with the middle term isomorphic
to $L,$ then $\mathrm{Ext}^{1}(M,N)_{L}$ can be viewed  as a
constructible subset of $\mathrm{Ext}^1(M,N)$ by identifying
$\mathrm{Ext}^1(M,N)$ and $E(M,N).$  Define
$$
\mathrm{Ext}^1(M,N)_{\mo}=\{[0\rightarrow N\rightarrow
L\rightarrow M\rightarrow 0]\in \mathrm{Ext}^1(M,N)\setminus\{0\}
\mid L\in \mo\}
$$
where the set $\mo$ is a $G_{\textbf d_1+\textbf d_2}$-invariant
constructible subset of $\bbe_{\textbf d_1+\textbf d_2}$ (see
\cite{XXZ2006} or \cite{GLS2006}). It can be identified with
$$
E(M,N)_{\mo}=\{m\in E(M,N)\mid L(m)\in \mo \}
$$
which is constructible and $\bbc^{*}$-invariant, see
\cite{GLS2006}. Hence,  $\mathrm{Ext}^1(M,N)_{\mo}$ can be viewed
as a $\bbc^{*}$-invariant constructible subset of
$\mathrm{Ext}^1(M,N)\setminus \{0\}.$ Let $\textbf e_1, \textbf
e_2$ be two dimension vectors.

\subsection{High order associativity}\label{highorder}
Let $M, N\in\mathrm{ mod}\bbc Q$ and $\tau$ be the
Auslander-Reiten translation on $\mathrm{ mod}\bbc Q$.  We assume
that $M$ contains no projective summands. Note that for any $X\in
\mathrm{ mod}\bbc Q,$ there is a decomposition of $\bbc Q$-modules
$$
X\cong \tau(\tau^{-1}X)\oplus X/\tau(\tau^{-1}X)
$$
with $X/\tau(\tau^{-1}X)$ isomorphic to an injective $\bbc
Q$-module. For dimension vectors $\textbf d_1,\textbf d_2$ and
$\textbf e_1,\textbf e_2,$ we consider the sets
$$\hspace{-2cm}\mathrm{DEF}^{\textbf d_1,
\textbf d_2}_{\textbf e_1, \textbf e_2}(N,\tau
M)=\{(g,Y_1,X_1)\mid g\in \mathrm{Hom}_{\bbc Q}(N,\tau M),
\textbf{dim} \mathrm{ker}g=\textbf d_1,$$$$\hspace{1cm}
\textbf{dim} \tau^{-1}(\mathrm{coker}g) =\textbf d_2, V_1\in
\mathrm{Gr}_{\textbf e_1}(\Upsilon_0\Upsilon^2(g)), U_1\in
\mathrm{Gr}_{\textbf e_2}(\phi^{-}\Omega_0\Omega^2(g))\}$$ and
$\mathrm{DFE}_{\textbf e_1,\textbf e_2}(N,\tau M)=$
$$\{(N_1,M_1, g')\mid  N_1\in \mathrm{Gr}_{\textbf e_1}(N), M_1\in \mathrm{Gr}_{\textbf e_2}(M), g'\in \hom_{\bbc Q}(N/N_1,\tau M_1) \}.$$
Here, $\Upsilon_0, \Upsilon^2, \Omega_0, \Omega^2$ and $\phi^{-}$
were defined as in Section 2. By definition,
$\Upsilon_0\Upsilon^2(g)\in \bbe_{\textbf d_1}$,
$\phi^{-}\Omega_0\Omega^2(g)\in \bbe_{\textbf d_2}$  and
$\Upsilon_0\Upsilon^2(g)\cong \mathrm{ker}g$,
$\phi^{-}\Omega_0\Omega^2(g)\cong \tau^{-1}(\mathrm{coker}g)$. We
set $U=\phi^{-}\Omega_0\Omega^2(g)$ and
$V=\Upsilon_0\Upsilon^2(g)$.
\begin{remark}
Let's explain the notations for these sets. The letter ``D'' means
``derived''. The letters ``E'' and ``F'' stand for extension and
flag, respectively. Let $M$ and $N$ be two indecomposable $\bbc
Q$-modules. Then (for example, see \cite{BMRRT} or
\cite{Hubery2005})
$$
\ext^1_{\mathcal{C}(Q)}(N,M)\cong \ext^1_{\bbc Q }(N,M)\oplus
\hom_{\bbc Q}(N, \tau M)
$$
and if $\ext^1_{\bbc Q }(N,M)=0,$ then any $g\in \hom_{\bbc Q}(N,
\tau M)$ induces an extension of $M$ by $N$ in the cluster category
$\mc(Q)$ as follows
$$
M\rightarrow \mathrm{ker}g\oplus \mathrm{coker}g[-1]\rightarrow
N\xrightarrow{g}\tau M.
$$
\end{remark}

Recall that each principal $\bbc^{*}$-bundle is locally trivial in
the Zariski topology. Let $\pi: P\rightarrow Q$ be such a bundle.
Then $(\pi, Q)$ is a geometric quotient for the free action of
$\bbc^{*}$ on $P$ (see \cite{serre} and \cite{GLS2006}). In the
following, we will write $\mathbb{P}x$ for the $\bbc^*$-orbit of
$x$ if $x$ belongs to a principal $\bbc^*$-bundle.

Let $\hom(N,\tau M)(\textbf d_1, \textbf d_2)$ be the subset of
$\hom(N,\tau M)$ consisting of the morphism $g$ with $\textbf{dim}
\mathrm{ker}g=\textbf d_1, \textbf{dim} \tau^{-1}(\mathrm{coker}g)
=\textbf d_2$. By Corollary \ref{partition}, we have finite
subsets $S(\textbf d_1)$ and $S(\textbf d_2)$ of $\bbe_{\textbf
d_1}$ and $\bbe_{\textbf d_2},$ respectively such that
$$
\bbe_{\textbf d_1}=\bigsqcup_{V\in S(\textbf d_1)}\str{V}, \quad
\bbe_{\textbf d_2}=\bigsqcup_{U\in S(\textbf d_2)}\str{U}.
$$
It induces a finite partition
$$
\hom(N,\tau M)(\textbf d_1, \textbf d_2)=\bigsqcup_{V\in S(\textbf
d_1), U\in S(\textbf d_2), I }\hom(N,\tau M)_{\str{V}\oplus
\str{U}\oplus I[-1]},
$$
where $\hom(N,\tau M)_{\str{V}\oplus \str{U}\oplus I[-1]}$ is
$$
\{g\in \hom(N,\tau M)(\textbf d_1, \textbf d_2)\mid
\Upsilon_0\Upsilon^2(g)\in \str{V}, \Omega_0\Omega^2(g)= \tau
U'\oplus I,
$$
$$
\mbox{ for some }U'\in \str{U}, I \mbox{ is an injective } \bbc
Q\mbox{-module}\}.
$$
Note that $\Omega_0\Omega^2(g)\cong \mathrm{coker}g.$

There is a natural $\bbc^{*}$-action on $\hom(N,\tau M)(\textbf
d_1, \textbf d_2)^{*}:=\hom(N,\tau M)(\textbf d_1, \textbf
d_2)\setminus \{0\}$ with a principal $\bbc^{*}$-bundle:
$$\hom(N,\tau M)(\textbf d_1, \textbf d_2)^{*}\rightarrow
\mathbb{P}\hom(N,\tau M)(\textbf d_1, \textbf d_2).$$  Thus by
considering the trivial action of $\bbc^{*}$ on
$\mathrm{Gr}_{\textbf e_1,\textbf e_2}(\textbf d_1,\textbf
d_2):=\mathrm{Gr}_{\textbf e_1}(\bbe_{\textbf d_1})\times
\mathrm{Gr}_{\textbf e_2}(\bbe_{\textbf d_2}),$ we obtain a new
principal $\bbc^{*}$-bundle (similar to \cite[Section
5.4]{GLS2006}):
$$
\pi: \hom(N,\tau M)(\textbf d_1, \textbf d_2)^{*}\times
\mathrm{Gr}_{\textbf e_1,\textbf e_2}(\textbf d_1,\textbf
d_2)\rightarrow \hom(N,\tau M)(\textbf d_1, \textbf
d_2)^{*}\times^{\bbc^{*}} \mathrm{Gr}_{\textbf e_1,\textbf
e_2}(\textbf d_1,\textbf d_2).
$$
We note that the action of $\bbc^{*}$ on $\hom(N,\tau M)(\textbf
d_1, \textbf d_2)^{*}\times \mathrm{Gr}_{\textbf e_1,\textbf
e_2}(\textbf d_1,\textbf d_2)$ is free. The set
$\mathrm{DEF}^{\textbf d_1, \textbf d_2}_{\textbf e_1,\textbf
e_2}(N,\tau M)$ is its $\bbc^{*}$-stable constructible subset. This
implies that
$$\mathbb{P}\mathrm{DEF}^{\textbf d_1, \textbf d_2}_{\textbf e_1,\textbf e_2}(N,\tau M):=\pi(\mathrm{DEF}_{\textbf e_1,\textbf e_2}(N,\tau M))$$  is
again a principal $\bbc^{*}$-bundle and
$(\pi,\mathbb{P}\mathrm{DEF}_{\textbf e_1,\textbf e_2}(N,\tau M))$
is a geometric quotient for the action of $\bbc^{*}$ on
$\mathrm{DEF}_{\textbf e_1,\textbf e_2}(N,\tau M)$ (similar to
\cite[Section 5.4]{GLS2006}). We have a natural projection:
$$p: \mathbb{P}\mathrm{DEF}^{\textbf d_1, \textbf d_2}_{\textbf e_1,\textbf e_2}(N,\tau M)\rightarrow
\mathbb{P}\hom(N,\tau M)(\textbf d_1, \textbf d_2).$$ There is a
finite partition
$$
\mathbb{P}\hom(N,\tau M)(\textbf d_1, \textbf d_2)=\bigsqcup_{V\in
S(\textbf d_1), U\in S(\textbf d_2), I }\mathbb{P}\hom(N,\tau
M)_{\str{V}\oplus \str{U}\oplus I[-1]}
$$
where $\mathbb{P}\hom(N,\tau M)_{\str{V}\oplus \str{U}\oplus
I[-1]}$ is the set
$$
\{\mathbb{P}g\in \mathbb{P}\hom(N,\tau M)(\textbf d_1, \textbf
d_2)\mid \Upsilon_0\Upsilon^2(g)\in \str{V},  \Omega_0\Omega^2(g)=
\tau U'\oplus I,
$$
$$
\mbox{ for some }U'\in \str{U}, I \mbox{ is an injective } \bbc
Q\mbox{-module}\}.
$$

For any $\mathbb{P}g\in \mathbb{P}\hom(N,\tau M)(\textbf d_1,
\textbf d_2)_{\str{V}\oplus \str{U}\oplus I[-1]}$, the Euler
characteristic of the fibre $p^{-1}(\mathbb{P}g\in
\mathbb{P}\hom(N,\tau M)(\textbf d_1, \textbf d_2))$ is
$\chi(\mathrm{Gr}_{\textbf e_1}(V))\cdot \chi(\mathrm{Gr}_{\textbf
e_2}(U))$. By Proposition \ref{Euler} we obtain the following
lemma.
\begin{lemma}\label{higheuler1}
For fixed dimension vector $\textbf d,$ we have
$$
\hspace{-5cm}\sum_{\textbf e_1+\textbf e_2+\textbf{dim} M-\textbf
d_2=\textbf d}\chi(\mathbb{P}\mathrm{DEF}^{\textbf d_1,\textbf
d_2}_{\textbf e_1,\textbf e_2}(N,\tau
M))$$$$\hspace{-1cm}=\sum_{\begin{subarray}{l}\textbf d_1,\textbf d_2,\textbf e_1,\textbf e_2,U,V, I;\\
\textbf e_1+\textbf e_2+\textbf{dim} M-\textbf d_2=\textbf d,
\\U\in S(\textbf d_2),V\in
S(\textbf d_1)\end{subarray}}\chi(\mathbb{P}\mathrm{Hom}(N,\tau
M)_{\str{V}\oplus \str{U}\oplus I[-1]})\chi(\mathrm{Gr}_{\textbf
e_1}(V)\chi(\mathrm{Gr}_{\textbf e_2}(U))
$$
\end{lemma}

There is also a free action of $\bbc^{*}$ on $\mathrm{DFE}_{\textbf
e_1,\textbf e_2}(N,\tau M)$ defined by
$$
t.(N_1,M_1,g)=(N_1,M_1,t.g)
$$
for any $t\in \bbc^{*}$ and $(N_1,M_1,g)\in \mathrm{DFE}_{\textbf
e_1,\textbf e_2}(N,\tau M).$ The orbit space is denoted by
$\mathbb{P}\mathrm{DFE}_{\textbf e_1,\textbf e_2}(N,\tau M).$

Consider a natural projection
$$q: \mathbb{P}\mathrm{DFE}_{\textbf e_1,\textbf e_2}(N, \tau M)\rightarrow
\mathrm{Gr}_{\textbf e_1}(N)\times \mathrm{Gr}_{\textbf e_2}(M).$$
Define $\str{(N_1 ,M_1)}$ to be
$$
\{(N'_1,M'_1)\in \mathrm{Gr}_{\textbf e_1}(N)\times
\mathrm{Gr}_{\textbf e_2}(M)\mid
\chi(\mathbb{P}\hom(N/N'_1,M'_1))=\chi(\mathbb{P}\hom(N/N_1,\tau
M_1))\}.
$$
We note that the notation is different from the Euler form
$\left<N_1, M_1\right>$ of $N_1$ and $M_1.$ Since
$q_{*}(1_{\mathbb{P}\mathrm{DFE}_{\textbf e_1,\textbf e_2}(N,\tau
M)})$ is a constructible function on $\mathrm{Gr}_{\textbf
e_1}(N)\times \mathrm{Gr}_{\textbf e_2}(\tau M)$ by Theorem
\ref{Joyce}, $\str{(N_1 ,M_1)}$ is a constructible subset and
there exists a finite subset $R(\textbf e_1,\textbf e_2)$ of
$\mathrm{Gr}_{\textbf e_1}(N)\times Gr_{\textbf e_2}(M)$ such that
$$
\mathrm{Gr}_{\textbf e_1}(N)\times \mathrm{Gr}_{\textbf
e_2}(M)=\bigsqcup_{(N_1,M_1)\in R(\textbf e_1,\textbf
e_2)}\str{(N_1,M_1)}.
$$
Hence, by Proposition \ref{Euler}, we have the following lemma.
\begin{lemma}\label{higheuler2}
For fixed dimension vector $\textbf d,$ we have
$$
\sum_{\textbf e_1,\textbf e_2;\textbf e_1+\textbf e_2=\textbf
d}\chi(\mathbb{P}\mathrm{DFE}_{\textbf e_1,\textbf e_2}(N,\tau
M))=$$$$\sum_{\textbf e_1,\textbf e_2;\textbf e_1+\textbf
e_2=\textbf d}\sum_{(N_1,M_1)\in R(\textbf e_1,\textbf
e_2)}\chi(\str{(N_1,M_1)})\chi(\mathbb{P}\mathrm{Hom}(N/N_1, \tau
M_1)).
$$
\end{lemma}
\bigskip

Now, we can compare $\mathrm{DEF}^{\textbf d_1, \textbf
d_2}_{\textbf e_1,\textbf e_2}(N,\tau M))$ with
$\mathrm{DFE}_{\textbf e_1,\textbf e_2}(N,\tau M))$.  Let
$(g,V_1,U_1)\in DEF^{\textbf d_1, \textbf d_2}_{\textbf
e_1,\textbf e_2}(N,\tau M).$ Then we have a long exact sequence
$$
0\rightarrow V\rightarrow N\xrightarrow{g} \tau M\rightarrow \tau
U\oplus I\rightarrow 0
$$
where $V=\Upsilon_0\Upsilon^2(g)$ and
$U=\phi^-\Omega_0\Omega^2(g)$. It is clear that $V_1\in
\mathrm{Gr}_{\textbf e_1}(N)$. By Proposition \ref{Kernelcokernel}
and Corollary \ref{grassmanian-submodule}, we have a morphism of
varieties $\mathrm{Gr}_{\textbf e_1}(N)\rightarrow
\bbe_{\textbf{dim}N-\textbf e_1}$ sending any submodule $N_1$ to a
$\bbc Q$-module isomorphic to the quotient module $N/N_1$. Let
$U^{*}$ be the pullback of the following diagram
$$
\xymatrix{
 0\ar[r]&\Omega^2(g) \ar[r]\ar@{=}[d]&U^{*}\ar[r]\ar[d]&
 U_1\ar[d]\ar[r]&0\\0\ar[r]&
\Omega^2(g) \ar[r]&M\ar[r]& U\ar[r]&0}
$$
Note that $\Omega^2(g)=\mathrm{Im}g$. Then $U^{*}\in
\mathrm{Gr}_{\textbf{dim}M-\textbf{dim}U+\textbf{dim}U_1}(M)$ and
we have the commutative diagram
$$
\xymatrix{ 0\ar[r]& V\ar[r]\ar[d]& N\ar[r]^{g}\ar[d]& \tau M\ar[r]&
\tau U\oplus I\ar[r]& 0\\
0\ar[r]& V/V_1\ar[r]& N/Y_1\ar[r]^{g'}& \tau U^{*}\ar[r]\ar[u] &
\tau U_1\oplus I\ar[r]\ar[u]& 0}
$$
Hence, for fixed dimension vector $\textbf d$, by Proposition
\ref{Kernelcokernel}, Corollary \ref{grassmanian-submodule} and
\ref{AR-grassmanian}, we have a morphism:
$$
\Gamma: \bigsqcup_{\begin{subarray}{l}\,\,\,\,\,\,\textbf d_1,\textbf d_2,\textbf e_1,\textbf e_2;\\
\textbf e_1+\textbf e_2+\textbf{dim} M-\textbf d_2=\textbf d,
\end{subarray} }\mathrm{DEF}^{\textbf d_1,\textbf
d_2}_{\textbf e_1,\textbf e_2}(N,\tau M)\rightarrow
\bigsqcup_{\textbf e'_1,\textbf e'_2; \textbf e'_1+\textbf
e_2=\textbf d}\mathrm{DFE}_{\textbf e'_1,\textbf e'_2}(N,\tau M)
$$
mapping $(g, V_1,U_1)$ to $(V_1, U^{*}, g')$. Conversely, we have an
inverse morphism:
$$
\Gamma': \bigsqcup_{\textbf e'_1,\textbf e'_2}\mathrm{DFE}_{\textbf
e'_1,\textbf e'_2}(N,\tau M)\rightarrow \bigsqcup_{\begin{subarray}{l}\,\,\,\,\,\,\textbf d_1,\textbf d_2,\textbf e_1,\textbf e_2;\\
\textbf e_1+\textbf e_2+\textbf{dim} M-\textbf d_2=\textbf d,
\end{subarray} }\mathrm{DEF}^{\textbf d_1,\textbf
d_2}_{\textbf e_1,\textbf e_2}(N,\tau M).
$$
The action of $\bbc^{*}$ induces the homeomorphism
$$
\mathbb{P}\Gamma: \bigsqcup_{\begin{subarray}{l}\,\,\,\,\,\,\textbf d_1,\textbf d_2,\textbf e_1,\textbf e_2;\\
\textbf e_1+\textbf e_2+\textbf{dim} M-\textbf d_2=\textbf d,
\end{subarray} }\mathbb{P}\mathrm{DEF}^{\textbf d_1,\textbf
d_2}_{\textbf e_1,\textbf e_2}(N,\tau M)\rightarrow
\bigsqcup_{\textbf e'_1,\textbf e'_2; \textbf e'_1+\textbf
e_2=\textbf d}\mathbb{P}\mathrm{DFE}_{\textbf e'_1,\textbf
e'_2}(N,\tau M).
$$
By Lemma \ref{higheuler1} and \ref{higheuler2}, The above
homeomorphism induces the following proposition referred to as the
high order associativity.
\begin{Prop}\label{associativity2}
With the above notations, for fixed dimension vector $\textbf d$, we
have
$$
\sum_{\begin{subarray}{l}\textbf d_1,\textbf d_2,\textbf e_1,\textbf e_2,U,V, I;\\
\textbf e_1+\textbf e_2+\textbf{dim} M-\textbf d_2=\textbf d,
\\U\in S(\textbf d_2),V\in
S(\textbf d_1)\end{subarray}}\chi(\mathbb{P}\mathrm{Hom}(N,\tau
M)_{\str{V}\oplus \str{U}\oplus I[-1]})\chi(\mathrm{Gr}_{\textbf
e_1}(V))\chi(\mathrm{Gr}_{\textbf e_2}(U))
$$
$$
=\sum_{\textbf e'_1,\textbf e'_2;\textbf e'_1+\textbf e'_2=\textbf
d}\sum_{(N_1,M_1)\in R(\textbf e'_1,\textbf
e'_2)}\chi(\str{(N_1,M_1)})\chi(\mathbb{P}\mathrm{Hom}(N/N_1, \tau
M_1)).
$$
\end{Prop}

\section{main theorem and proof}

\subsection{The main theorem} We introduce some notations.
For any $\bbc Q$-module $M$ and projective $\bbc Q$-module $P$,
let $I=\mathrm{DHom}_{\bbc Q}(P,\bbc Q).$  By Corollary
\ref{partition}, we have the following finite partitions:
$$
\hom(M, I)=\bigsqcup_{I', V\in S(\textbf
d_1(I'))}\hom(M,I)_{\str{V}\oplus I'[-1]},
$$
$$
\hom(P, M)=\bigsqcup_{P', U\in S(\textbf
d_2(P'))}\hom(P,M)_{P'[1]\oplus\str{U}},
$$
where $\textbf
d_1(I')=\textbf{dim}I+\textbf{dim}I'-\textbf{dim}M$, $\textbf
d_2(P')=\textbf{dim}P+\textbf{dim}P'-\textbf{dim}M$,
$\bbe_{\textbf d_1}=\bigsqcup_{V\in S(\textbf d_1)}\str{V}$,
$\bbe_{\textbf d_2}=\bigsqcup_{U\in S(\textbf d_2)}\str{U}$,
$\hom(M,I)_{\str{V}\oplus I'[-1]}=\{f\in \hom(P,M)\mid
\Upsilon_0\Upsilon^2(f)=V', \Omega_0\Omega^2(f)=I' \mbox{ for some
} V'\in \str{V}\} $, and $ \hom(P,M)_{P'[1]\oplus\str{U}}=\{g\in
\hom(P,M)\mid \Upsilon_0\Upsilon^2(g)=P', \Omega_0\Omega^2(g)=U'
\mbox{ for some } U'\in \str{U}\} $. Note that
$\Upsilon_0\Upsilon^2(f)\cong \mathrm{ker}f$ and
$\Omega_0\Omega^2(f)\cong \mathrm{coker}f.$

The following theorem generalizes the cluster multiplication
theorem for finite type in \cite{CK2005} and affine type in
\cite{Hubery2005} and is referred to as the cluster multiplication
theorem for arbitrary type.

\begin{theorem}\label{clustertheorem}
Let $Q$ be an acyclic quiver. Then

\nd (1) for any $\bbc Q$-modules $M,N$ such that $M$ contains no
projective summand, we have
$$\hspace{0cm}\mathrm{dim}_{\bbc}\mathrm{Ext}^1_{\bbc Q}(M,N)X_{M} X_{N} =\sum_{L\in S(\textbf e)}\chi(\mathbb{P}\mathrm{Ext}^{1}_{\bbc Q}(M,N)_{\str{L}})X_{L}$$
$$
+\sum_{I, \textbf d_1,\textbf d_2}\sum_{ V\in S(\textbf d_1),U\in
S(\textbf d_2)}\chi(\mathbb{P}\mathrm{Hom}_{\bbc Q}(N,\tau
M)_{\str{V}\oplus \str{U}\oplus I[-1]})X_{U}X_{V}x^{\textbf{dim}
\mathrm{soc}I}
$$
where $\textbf e=\textbf{dim} M+\textbf{dim} N$.

\nd (2) for any $\bbc Q$-module $M$ and projective $\bbc Q$-module
$P$, we have
 $$\mathrm{dim}_{\bbc}\mathrm{Hom}_{\bbc Q}(P,M)X_{M}x^{\textbf{dim}top(P)}=\sum_{I',V\in S(\textbf d_1(I'))}
 \chi(\mathbb{P}\mathrm{Hom}_{\bbc Q}(M,I)_{\str{V}\oplus I'[-1]
})X_{V}x^{\textbf{dim}\mathrm{soc}I'}$$$$+\sum_{P', U\in S(\textbf
d_2(P'))}\chi(\mathbb{P}\mathrm{Hom}_{\bbc Q}(P,M)_{P'[1]\oplus
\str{U}})X_{U}x^{\textbf{dim}top(P')}
 $$
where $top(P)=P/radP$, $I=\mathrm{DHom}_{\bbc Q}(P,\bbc Q),$ $I'$
is injective, and $P'$ is projective.
\end{theorem}
\begin{proof}
We set
$$
\Sigma_2:=\sum_{I, \textbf d_1,\textbf d_2,}\sum_{ V\in S(\textbf
d_1),U\in S(\textbf d_2)}\chi(\mathbb{P}\mathrm{Hom}_{\bbc
Q}(N,\tau M)_{\str{V}\oplus \str{U}\oplus
I[-1]})X_{U}X_{V}x^{\textbf{dim} \mathrm{soc}I}
$$
By definition of $X_{M},$ the sum is
$$
\sum_{ I,\textbf d_1,\textbf d_2,\textbf e_1,\textbf e_2,
  V\in S(\textbf d_1),U\in
S(\textbf d_2) }\chi(\mathbb{P}\mathrm{Hom}_{\bbc Q}(N,\tau
M)_{\str{V}\oplus \str{U}\oplus
I[-1]})\cdot$$$$\chi(\mathrm{Gr}_{\textbf
e_1}(V))\chi(\mathrm{Gr}_{\textbf e_2}(U)x^{(\textbf e_1+\textbf
e_2)R+(\textbf{dim} V-\textbf e_1+\textbf{dim} U-\textbf
e_2)R^{tr}-(\textbf{dim}V+\textbf{dim} U)+\textbf{dim}
\mathrm{soc}I}.
$$
By Lemma \ref{Huberylemma1}, we have
$$
\hspace{-3.5cm}(\textbf{dim} V+\textbf{dim}
U)R^{tr}-(\textbf{dim}V+\textbf{dim} U)+\textbf{dim} \mathrm{soc}I
$$
$$
\hspace{-3.1cm}=(\textbf{dim} V+\textbf{dim}
U)R^{tr}-(\textbf{dim}V+\textbf{dim} U)+\textbf{dim}I(1-R^{tr})
$$
$$
=(\textbf{dim}\tau
U+\textbf{dim}I-\textbf{dim}V)-(\textbf{dim}\tau
U+\textbf{dim}I-\textbf{dim}V)R^{tr}+\textbf{dim}U (R^{tr}-R)
$$
$$
\hspace{-2.3cm}=(\textbf{dim}\tau
M-\textbf{dim}N)-(\textbf{dim}\tau
M-\textbf{dim}N)R^{tr}+\textbf{dim}U (R^{tr}-R)
$$
$$
\hspace{-3.1cm}=\textbf{dim}\tau
M(1-R^{tr})-\textbf{dim}N+\textbf{dim}N R^{tr}+\textbf{dim}U
(R^{tr}-R)
$$
$$
\hspace{-2.3cm}=(\textbf{dim}M-\textbf{dim}U)R+(\textbf{dim}N+\textbf{dim}U)R^{tr}-(\textbf{dim}M+\textbf{dim}N)
$$
Then, the exponent of $x$ is
$$
(\textbf e_1+\textbf
e_2+\textbf{dim}M-\textbf{dim}U)R+(\textbf{dim}N+\textbf{dim}U-\textbf
e_1-\textbf e_2)R^{tr}-(\textbf{dim}M+\textbf{dim}N).
$$
Hence, we have
$$
\Sigma_2=\sum_{\textbf d}\sum_{\begin{subarray}{r}I,\textbf
d_1,\textbf d_2,\textbf e_1,\textbf e_2;\\\textbf e_1+\textbf
e_2+\textbf{dim}M-\textbf d_2=\textbf d\end{subarray}}
  \sum_{V\in S(\textbf d_1),U\in
S(\textbf d_2) }\chi(\mathbb{P}\mathrm{Hom}_{\bbc Q}(N,\tau
M)_{\str{V}\oplus \str{U}\oplus I[-1]})\cdot
$$
$$
\chi(\mathrm{Gr}_{\textbf e_1}(V))\chi(\mathrm{Gr}_{\textbf
e_2}(U)x^{\textbf d R+(\textbf e-\textbf d)R^{tr}-\textbf e}
$$
where $\textbf e=\textbf{dim}M+\textbf{dim}N.$ Using Proposition
\ref{associativity2}, we obtain that $\Sigma_2$ is equal to
$$
\sum_{\textbf d}\sum_{\textbf e'_1,\textbf e'_2;\textbf
e'_1+\textbf e'_2=\textbf d}\sum_{(N_1,M_1)\in R(\textbf
e'_1,\textbf
e'_2)}\chi(\str{(N_1,M_1)})\chi(\mathbb{P}\mathrm{Hom}_{\bbc
Q}(N/N_1, \tau M_1)) x^{\textbf d R+(\textbf e-\textbf
d)R^{tr}-\textbf e}.
$$

Now we set
$$
\Sigma_1:=\sum_{L\in S(\textbf
e)}\chi(\mathbb{P}\mathrm{Ext}^{1}_{\bbc Q}(M,N)_{\str{L}})X_{L}
$$
We consider
$$\mathrm{EF}_{\textbf d}(M,N)=\{(\varepsilon,
L_1)\mid  \varepsilon\in \ext^{1}_{\bbc Q}(M,N)_{L}\setminus\{0\},
L_1\in \mathrm{Gr}_{\textbf d}(L)\}.$$ As a vector space,
$L=M\oplus N.$ Define
$$t.(m,n)=(m,t.n)$$ for any $(m,n)\in M\oplus N$ and $t\in \bbc^{*}$. This induces the
action of $\bbc^{*}$ on $L_1.$ So we have an $\bbc^{*}$-action on
$\mathrm{EF}_{\textbf d}(M,N)$ (\cite{GLS2006}).  As the
discussion in Section 2.5, the $\bbc^{*}$-action induces the
geometric quotient $\mathbb{P}\mathrm{EF}_{\textbf d}(M,N).$ The
projection
$$\mathbb{P}\mathrm{EF}_{\textbf d}(M,N)\rightarrow \mathbb{P}\ext^{1}_{\bbc Q}(M,N)$$ has the
fibre $\{(\mathbb{P}\varepsilon, L_1)\mid L_1\in
\mathrm{Gr}_{\textbf d}(L)\}$ for any $\mathbb{P}\varepsilon\in
\mathbb{P}\ext^{1}_{\bbc Q}(M,N)_{L}.$ By Theorem \ref{Euler} and
Corollary \ref{partition}, this implies
$$
\chi(\mathbb{P}\mathrm{EF}_{\textbf d}(M,N))=\sum_{L\in S(\textbf
e)}\chi(\mathbb{P}\mathrm{Ext}^{1}_{\bbc
Q}(M,N)_{\str{L}})\chi(\mathrm{Gr}_{\textbf d}(L))
$$
and
$$
\Sigma_1=\sum_{\textbf d}\chi(\mathbb{P}\mathrm{EF}_{\textbf
d}(M,N))x^{\textbf d R+(\textbf e-\textbf d)R^{tr}-\textbf e}.
$$
On the other hand, we have a natural morphism
$$
\Psi: \mathrm{EF}_{\textbf d}(M,N)\rightarrow \bigsqcup_{\textbf
e'_1,\textbf e'_2;\textbf e'_1+\textbf e'_2=\textbf
d}\mathrm{Gr}_{\textbf e'_1}(M)\times \mathrm{Gr}_{\textbf e'_2}(N)
$$
mapping $(\varepsilon=[(f,g)], L_1)$ to $(g(L_1), f^{-1}(L_1)).$
Here $\varepsilon$ is the equivalence class of the exact sequence
$$
0\rightarrow N\xrightarrow{f} L\xrightarrow{g}M\rightarrow 0
$$
The morphism $\Psi$ induces a morphism
$$
\mathbb{P}\Psi: \mathbb{P}\mathrm{EF}(M,N):=\bigsqcup_{\textbf
d}\mathbb{P}\mathrm{EF}_{\textbf d}(M,N)\rightarrow
\bigsqcup_{\textbf e'_1,\textbf e'_2}\mathrm{Gr}_{\textbf
e'_1}(M)\times \mathrm{Gr}_{\textbf e'_2}(N).
$$
Let's compute the fibre for any $(M_1,N_1)\in \mathrm{Gr}_{\textbf
e'_1}(M)\times \mathrm{Gr}_{\textbf e'_2}(N).$ Consider the map
$$
\beta': \ext^{1}(M,N)\oplus \ext^{1}_{\bbc Q}(M_1,N_1)\rightarrow
\ext^{1}_{\bbc Q}(M_1,N)
$$
sending $(\varepsilon,\varepsilon')$ to
$\varepsilon_{M_1}-\varepsilon'_{N}$ where $\varepsilon_{M_1}$ and
$\varepsilon'_{N}$ are induced by including $M_1\subseteq M$ and
$N_1\subseteq N,$ respectively as follows:
$$
\xymatrix{
\varepsilon_{M_1}:&0\ar[r]&N\ar[r]\ar@{=}[d]&L_1\ar[r]\ar[d]&M_1\ar[r]\ar[d]&0\\
\varepsilon: & 0\ar[r]&N\ar[r]&L\ar[r]^-{\pi}&M\ar[r]&0 }
$$
where $L_1$ is the pullback, and
$$
\xymatrix{
\varepsilon':&0\ar[r]&N_1\ar[r]\ar[d]&L'\ar[r]\ar[d]&M_1\ar[r]\ar@{=}[d]&0\\
\varepsilon'_N: & 0\ar[r]&N\ar[r]&L'_1\ar[r]&M_1\ar[r]&0 }
$$
where $L'_1$ is the pushout. It is clear that
$\varepsilon,\varepsilon'$ and $M_1,N_1$ induce the inclusions
$L_1\subseteq L$ and $L'\subseteq L'_1.$ Considering the map
$$
p_0: \ext^{1}_{\bbc Q}(M,N)\oplus \ext_{\bbc
Q}^{1}(M_1,N_1)\rightarrow \ext^{1}_{\bbc Q}(M,N),
$$
we have (\cite[Lemma
2.4.2]{GLS2006})$$(\mathbb{P}\Psi)^{-1}(M_1,N_1)=\{(\mathbb{P}\varepsilon,
L')\mid
\mathbb{P}\varepsilon\in\mathbb{P}(p_0(\mathrm{ker}\beta')), L'\in
F(\varepsilon,M_1,N_1) \}$$ where
$F(\varepsilon,M_1,N_1)=\{L'\subseteq L\mid \pi(L')=M_1,L'\cap
N=N_1\}$ is isomorphic to the affine space $\hom(M_1,N/N_1)$ or an
empty set (\cite[Lemma 7]{Hubery2005}, see also \cite[Lemma
3.3.1]{GLS2006} for a similar discussion). By the 2-Calabi-Yau
property (Auslander-Reiten formula) $\ext^{1}(M,N)\simeq
D\hom(N,\tau M),$ we can consider the dual of $\beta'$ which is
$$
\beta: \hom(N,\tau M_1)\rightarrow \hom(N,\tau M)\oplus
\hom(N_1,\tau M_1).
$$
By using the knowledge of bilinear form and orthogonality, we know
that as a vector space,
$$
(p_0(\mathrm{ker}\beta'))^{\perp}= \mathrm{Im}\beta\bigcap
\hom(N,\tau M)\simeq \hom(N/N_1,\tau M_1).
$$
Note that if $F(\varepsilon,M_1,N_1)$ is an empty set, then
$\mathbb{P}(p_0(\mathrm{ker}\beta'))$ is an empty set. In this
case,
$\mathrm{dim}_{\bbc}\mathrm{Ext}^1(M,N)=\chi(\mathbb{P}\mathrm{Hom}(N/N_1,\tau
M_1))$. Hence, we obtain
\begin{equation}\label{euleroffibre}
\chi((\mathbb{P}\Psi)^{-1}(M_1,N_1))=\mathrm{dim}_{\bbc}\mathrm{Ext}^1(M,N)-\chi(\mathbb{P}\mathrm{Hom}(N/N_1,\tau
M_1)).
\end{equation}
Now, using the partitions as in Proposition \ref{associativity2}, we
know
$$\mathrm{Gr}_{\textbf e'_1}(M)\times \mathrm{Gr}_{\textbf e'_2}(N)=\bigsqcup_{(N_1,M_1)\in R(\textbf e'_1,\textbf e'_2)}\str{(N_1,M_1)}.$$
Hence, according the Euler characteristic of the fibres in
\eqref{euleroffibre} and Theorem \ref{Euler}, we obtain $\Sigma_1$
is equal to
$$
\sum_{\textbf d}\sum_{\textbf e'_1,\textbf e'_2;\textbf
e'_1+\textbf e'_2=\textbf d}\sum_{(N_1,M_1)\in R(\textbf
e'_1,\textbf
e'_2)}\chi(\str{(N_1,M_1)})(\mathrm{dim}_{\bbc}\mathrm{Ext}^1(M,N)-\chi(\mathbb{P}\mathrm{Hom}_{\bbc
Q}(N/N_1, \tau M_1)))\cdot $$$$x^{\textbf d R+(\textbf e-\textbf
d)R^{tr}-\textbf e}
$$
Hence,
$$
\Sigma_1+\Sigma_2 =\mathrm{dim}_{\bbc}\mathrm{Ext}^1(M,N)X_MX_N.
$$
We complete the proof of the first assertion of the theorem. As for
the second part, we set
$$
T_1=\sum_{I', \textbf e_1, V; V\in S(\textbf d_1(I'))}
 \chi(\mathbb{P}\mathrm{Hom}_{\bbc Q}(M,I)_{\str{V}\oplus I'[-1]
})\chi(Gr_{\textbf e_1}(V))\cdot $$$$x^{\textbf e_1R+(\textbf
d_1(I')-\textbf e_1)R^{tr}-\textbf d_1(I')+\textbf{dim}
\mathrm{soc}I'}
$$
and
$$
T_2=\sum_{P',\textbf e_2, U; U\in S(\textbf
d_2(P'))}\chi(\mathbb{P}\mathrm{Hom}_{\bbc Q}(P,M)_{P'[1]\oplus
\str{U}})\chi(\mathrm{Gr}_{\textbf e_2}(V))\cdot$$$$x^{\textbf
e_2R+(\textbf d_2(P')-\textbf e_2)R^{tr}-\textbf
d_2(P')+\textbf{dim} top(P')}
$$
By the similar argument as in Corollary \ref{partition}, there is
a finite subset $R(\textbf e_1)$ of $\mathrm{Gr}_{\textbf e_1}(M)$
such that partition
$$
\mathrm{Gr}_{\textbf e_1}(M)=\bigsqcup_{M_1\in R(\textbf
e_1)}\{M_1\}
$$
where $\{M_1\}=\{W\in \mathrm{Gr}_{\textbf e_1}(M)\mid
\chi(\mathbb{P}\mathrm{Hom}(M/W,
I))=\chi(\mathbb{P}\mathrm{Hom}(M/M_1, I))\}$ is a constructible
subset of $\mathrm{Gr}_{\textbf e_1}(M)$. Note that
$$
 \{M_1\}=\{W\in \mathrm{Gr}_{\textbf e_1}(M)\mid
\chi(\mathbb{P}\mathrm{Hom}(P, W))=\chi(\mathbb{P}\mathrm{Hom}(P,
M_1))\}.
$$

 By using Proposition \ref{Euler}, we obtain that $T_1$ is equal to
$$
\sum_{\textbf e_1}\sum_{M_1\in R(\textbf e_1)}\chi(\{M_1\})\cdot
\chi(\mathbb{P}\mathrm{Hom}(M/M_1, I))x^{\textbf
e_1R+(\textbf{dim}M-\textbf
e_1)R^{tr}-\ud_M+\textbf{dim}\mathrm{soc}(I)}
$$
and $T_2$ is equal to
$$
\sum_{\textbf e_1}\sum_{M_1\in R(\textbf e_1)}\chi(\{M_1\})\cdot
\chi(\mathbb{P}\mathrm{Hom}(P, M_1))x^{\textbf
e_1R+(\textbf{dim}M-\textbf e_1)R^{tr}-\textbf{dim}
M+\textbf{dim}top(P)}.
$$
Since $\textbf{dim}\mathrm{soc}(I)= \textbf{dim}top(P)$ and
$$\chi(\mathbb{P}\mathrm{Hom}(P,
M_1))+\chi(\mathbb{P}\mathrm{Hom}(M/M_1,
I))=\chi(\mathbb{P}\mathrm{Hom}(P, M)),$$ we have
$$
T_1+T_2=\mathrm{dim}_{\bbc}\mathrm{Hom}(P, M)\sum_{\textbf
e}\chi(\mathrm{Gr}_{\textbf e}(M))x^{\textbf e R+(\textbf{dim}
M-\textbf e)R^{tr}-\textbf{dim} M+\textbf{dim} top(P)}
$$
$$
\hspace{-3.5cm}=\mathrm{dim}_{\bbc}\mathrm{Hom}_{\bbc
Q}(P,M)X_{M}x^{\textbf{dim}top(P)}.
$$
\end{proof}
\begin{remark}\label{final}
The proof of Theorem \ref{clustertheorem} only involves
Auslander-Reiten formula and the high order associativity. It
inspires us to look for an analog of the present cluster theorem
for hereditary categories with Serre duality. The simplest case is
as follows: if $Q$ is a Kronecker quiver, then we know $\md^b(Q)$
is derived equivalent to $\md^b(\mathrm{coh} \mathbb{P}^1)$. We
hope that the present approach can help us to found the cluster
multiplication for $\mathrm{coh} \mathbb{P}^1$. One of the
difficulties is to substitute the stacks for module varieties to
rewrite the results in this paper as done in \cite{Joyce} and
\cite{Schiffimann}. It will be interesting to compare these
cluster multiplication theorems.
\end{remark}

\begin{remark}
In Theorem \ref{clustertheorem}, the condition that $M$ contains
no projective summands is not essential. Let $M'=M\oplus P$ with
 the maximal projective summand $P$. Then we multiply two sides of
 the first equation in Theorem \ref{clustertheorem} by $X_P$ to
 obtain the equation involving $X_{M'}X_N.$
\end{remark}
Now we consider a particular case that $M$ is a non-projective
indecomposable $\bbc Q$-module and $N=\tau M.$ By the
Auslander-Reiten formula, there is an isomorphism of
$\mathrm{End}_{\bbc Q}(M)^{op}$-modules: $\ext^1_{\bbc Q}(M, \tau
M)\cong D\mathrm{End}_{\bbc Q}(\tau M)$. It induces the
isomorphisms:
\begin{equation}\label{AR new duality1}
\mathrm{soc}\mathrm{Ext}_{\bbc Q}^1(M, \tau M)\cong
D(\mathrm{End}_{\bbc Q}(\tau M)/\mathrm{rad}\mathrm{End}_{\bbc
Q}(\tau M)),
\end{equation} where $\mathrm{soc}\mathrm{Ext}_{\bbc
Q}^1(M, \tau M)$ is the socle of $\ext^1_{\bbc Q}(M, \tau M)$ as a
$\mathrm{End}_{\bbc Q}(M)^{op}$-module, and
\begin{equation}\label{AR new duality}
\mathrm{Ext}_{\bbc Q}^1(M, \tau M)/\mathrm{soc}\mathrm{Ext}_{\bbc
Q}^1(M, \tau M)\cong D(\mathrm{rad}\mathrm{End}_{\bbc Q}(\tau M)).
\end{equation}
The equations \eqref{AR new duality1} and \eqref{AR new duality} can
be viewed as variants of 2-Calabi-Yau property (Auslander-Reiten
formula).
 An extension $\varepsilon\in
\ext^1_{\bbc Q}(M, \tau M)$ is an Auslander-Reiten sequence if and
only if $\varepsilon\in \mathrm{soc}\mathrm{Ext}_{\bbc Q}^1(M, \tau
M).$ We denote by $L_0$ the middle term of $\varepsilon.$ In the
proof of Theorem \ref{clustertheorem}, we substitute
$\mathrm{soc}\mathrm{Ext}_{\bbc Q}^1(M, \tau M)$ or
$\mathrm{Ext}_{\bbc Q}^1(M, \tau M)/\mathrm{soc}\mathrm{Ext}_{\bbc
Q}^1(M, \tau M)$ for $\mathrm{Ext}_{\bbc Q}^1(M, \tau M)$ and
 the above variants \eqref{AR new duality1}
or \eqref{AR new duality} for Auslander-Reiten formula. Then we
obtain the following result (see \cite{CC} or Lemma 7 in
\cite{Hubery2005} for different proofs).

\begin{Prop}\label{ARcluster}
Let $Q$ be an acyclic quiver and $M$ be a non-projective
indecomposable $\bbc Q$-module. Then
$$\mathrm{dim}_{\bbc}\mathrm{Ext}^1_{\bbc Q}(M,\tau M)/\mathrm{soc}\mathrm{Ext}^1_{\bbc Q}(M,\tau M)X_{M} X_{\tau M} =\sum_{L\ncong L_0\in S(\textbf e)}\chi(\mathbb{P}\mathrm{Ext}^{1}_{\bbc Q}(M,\tau M)_{\str{L}})X_{L}$$
$$
+\sum_{I, \textbf d_1,\textbf d_2}\sum_{ V\in S(\textbf d_1),U\in
S(\textbf d_2)}\chi(\mathbb{P}\mathrm{rad}\mathrm{End}_{\bbc
Q}(\tau M)_{\str{V}\oplus \str{U}\oplus
I[-1]})X_{U}X_{V}x^{\textbf{dim} \mathrm{soc}I}
$$
and
$$
X_{M}X_{\tau M}=1+X_{L_0}
$$
where $\textbf e=\textbf{dim} M+\textbf{dim} \tau M$ and
$\mathbb{P}\mathrm{rad}\mathrm{End}_{\bbc Q}(\tau M)$ is the
quotient of $\mathrm{rad}\mathrm{End}_{\bbc Q}(\tau M)$ under the
free action of $\bbc^{*}$ and $L_0$ is the middle term of the
Auslander-Reiten sequence ending in $M$.
\end{Prop}
\begin{proof}
We only need to prove the second equation. It is equivalent to
prove that $\mathrm{dim}_{\bbc}\mathrm{soc}\mathrm{Ext}^1_{\bbc
Q}(M,\tau M)X_{M}X_{\tau M}$ is equal to
\begin{equation}\label{ARcluster2}
\chi(\mathbb{P}\mathrm{Ext}^1_{\bbc Q}(M,\tau
M)_{L_0})X_{L_0}+\chi(\mathbb{P}(\mathrm{End}_{\bbc Q}(\tau
M)/\mathrm{rad}\mathrm{End}_{\bbc Q}(\tau M))).
\end{equation}
We use the notation in the proof of Theorem \ref{clustertheorem},
and set
$$
\Sigma_1:=\chi(\mathbb{P}\mathrm{Ext}^{1}_{\bbc Q}(M,\tau
M)_{L_0})X_{L_0}
$$
and
$$\mathrm{EF}_{\textbf d}(M,\tau M)=\{(\varepsilon,
L_1)\mid  \varepsilon\in \ext^{1}_{\bbc Q}(M,N)_{L_0}, L_1\in
\mathrm{Gr}_{\textbf d}(L_0)\}.$$ The $\bbc^{*}$-action induces the
geometric quotient $\mathbb{P}\mathrm{EF}_{\textbf d}(M,\tau M).$ We
have
$$
\Sigma_1=\sum_{\textbf d}\chi(\mathbb{P}\mathrm{EF}_{\textbf
d}(M,\tau M))x^{\textbf d R+(\textbf e-\textbf d)R^{tr}-\textbf e}
$$
and a morphism
$$
\mathbb{P}\Psi: \mathbb{P}\mathrm{EF}(M,\tau
M):=\bigsqcup_{\textbf d}\mathbb{P}\mathrm{EF}_{\textbf d}(M,\tau
M)\rightarrow \bigsqcup_{\textbf e'_1,\textbf
e'_2}\mathrm{Gr}_{\textbf e'_1}(M)\times \mathrm{Gr}_{\textbf
e'_2}(\tau M).
$$
For any $(M_1, N_1)\in \mathrm{Gr}_{\textbf e'_1}(M)\times
\mathrm{Gr}_{\textbf e'_2}(\tau M),$ we consider the map
$$
\beta': \mathrm{soc}\ext^{1}(M,\tau M)\oplus \ext^{1}_{\bbc
Q}(M_1,N_1)\rightarrow \ext^{1}_{\bbc Q}(M_1,\tau M)
$$
and the map
$$
p_0: \mathrm{soc}\ext^{1}_{\bbc Q}(M,\tau M)\oplus \ext_{\bbc
Q}^{1}(M_1,N_1)\rightarrow \mathrm{soc}\ext^{1}_{\bbc Q}(M,\tau M),
$$
Then as in the proof of Theorem \ref{clustertheorem}, we have
$$(\mathbb{P}\Psi)^{-1}(M_1,N_1)=\{(\mathbb{P}\varepsilon,
L_1)\mid
\mathbb{P}\varepsilon\in\mathbb{P}(p_0(\mathrm{ker}\beta')),
L_1\in F(\varepsilon,M_1,N_1) \}$$ where
$F(\varepsilon,M_1,N_1)=\{L_1\subseteq L\mid \pi(L_1)=M_1,L_1\cap
N=N_1\}$ is isomorphic to the affine space $\hom(M_1,N/N_1)$ or a
empty set. By using the variant \eqref{AR new duality1} of
Auslander-Reiten formula $\mathrm{soc}\ext^{1}(M,N)\simeq
D(\mathrm{End}(\tau M)/\mathrm{radEnd}(M)),$ we can consider the
dual of $\beta'$
$$
\beta: \hom(\tau M,\tau M_1)\rightarrow \mathrm{End}(\tau
M)/\mathrm{radEnd}(M)\oplus \hom(N_1,\tau M_1)
$$
Then
$$
(p_0(\mathrm{ker}\beta'))^{\perp}= \mathrm{Im}\beta\bigcap
\mathrm{End}(\tau M)/\mathrm{radEnd}(M).
$$
which vanishes unless $N_1=0$ and $M_1=M$. Hence, we obtain

$$\chi((\mathbb{P}\Psi)^{-1}(M_1,N_1))=\left\{\begin{array}{cc}
\mathrm{dim}_{\bbc}\mathrm{soc}\mathrm{Ext}^1(M,N),& {\rm if}\,\  N_1\neq 0\,\ {\rm or}\,\ M_1\neq M,\\
0,&{\rm otherwise.}
\end{array}
\right.$$
This implies the equation \eqref{ARcluster2}.
\end{proof}

\subsection{An example}\label{example}
Let us illustrate Theorem \ref{clustertheorem} and Proposition
\ref{ARcluster} by the following example.  Let $Q$ be the
Kronecker quiver $\xymatrix{1\ar @<2pt>[r] \ar@<-2pt>[r]& 2}.$ Let
$S_1$ and $S_2$ be the simple modules associated to
vertices 1 and 2, respectively. Hence,
$$
R=\left(%
\begin{array}{cc}
  0 & 2 \\
  0 & 0 \\
\end{array}%
\right) \quad \mbox{and}\quad R'=\left(%
\begin{array}{cc}
  0 & 0 \\
  2 & 0 \\
\end{array}%
\right)
$$
and
$$
x_0:=X_{S_2}=x^{\mathrm{\underline{dim}}S_2R'-\mathrm{\underline{dim}}S_2}+x^{\mathrm{\underline{dim}}S_2R-\mathrm{\underline{dim}}S_2}=x_2^{-1}(1+x_1^2),
$$
$$
x_3:=X_{S_1}=x^{\mathrm{\underline{dim}}S_1R'-\mathrm{\underline{dim}}S_1}+x^{\mathrm{\underline{dim}}S_1R-\mathrm{\underline{dim}}S_1}=x_1^{-1}(1+x_2^2).$$

For $\lambda\in \mathbb{P}^1(\bbc),$ let $u_{\lambda}$ be the
regular representation $\xymatrix{\bbc\ar @<2pt>[r]^{1}
\ar@<-2pt>[r]_{\lambda}& \bbc}.$ By definition,
$$
X_{u_{\lambda}}=x^{(1,1)R'-(1,1)}+x^{(1,1)R-(1,1)}+x^{(0,1)R+(1,0)R'-(1,1)}=x_1x_2^{-1}+x_1^{-1}x_2
+x_1^{-1}x_2^{-1}.$$ Similarly, let $u_{\lambda(n)}$ ($n\geq 1$) be
the unique indecomposable regular $\bbc Q$-module with socle
$u_{\lambda}$ and length $n$. In particular, $u_{\lambda
(1)}=u_{\lambda}.$ Then $\mathrm{dim}_{\bbc}\ext^{1}(u_{\lambda},
u_{\lambda(n)})=1$ and for any $f\neq 0\in \hom(u_{\lambda(n)}, \tau
u_{\lambda}),$ we have an short exact sequence
$$
0\rightarrow u_{\lambda(n-1)}\rightarrow
u_{\lambda(n)}\xrightarrow{f} \tau u_{\lambda}\rightarrow 0.
$$
By using Theorem \ref{clustertheorem}, we have
$$
\mathrm{dim}_{\bbc}\ext^{1}(u_{\lambda},
u_{\lambda(n)})X_{u_{\lambda}}X_{u_{\lambda(n)}}=X_{u_{\lambda(n+1)}}-X_{u_{\lambda(n-1)}}.
$$

 It is clear that $X_{u_{\lambda(n)}}$ is irrelevant to the choice
of $\lambda\in \mathbb{P}^1(\bbc).$ We denote it by $r_n$. Set
$r_0=1$. Hence, we have
\begin{equation}\label{induction}
r_1=x_0x_3-x_1x_2 \quad \mbox{ and }\quad r_{n+1}=r_1r_n-r_{n-1}
\end{equation}
which has been shown as the elements of the dual semicanonical
canonical basis in \cite{CZ} and \cite{Zelevinsky}. For $n=2,$ it
is known that $\mathrm{dim}_{\bbc}\mathrm{Ext}^1_{\bbc
Q}(u_{\lambda(2)}, u_{\lambda(2)})=2$. The corresponding two
linear independent extensions are as follows:
$$
0\rightarrow u_{\lambda(2)}\rightarrow u_{\lambda(4)}\rightarrow
u_{\lambda(2)}\rightarrow 0$$ and
$$
0\rightarrow u_{\lambda(2)}\rightarrow u_{\lambda(1)}\oplus
u_{\lambda(3)}\rightarrow u_{\lambda(2)}\rightarrow 0.
$$
The latter is the Auslander-Reiten sequence. By using Theorem
\ref{clustertheorem}, we have
$$
\mathrm{dim}_{\bbc}\ext^{1}(u_{\lambda(2)},
u_{\lambda(2)})X_{u_{\lambda(2)}}X_{u_{\lambda(2)}}=X_{u_{\lambda(4)}}+X_{u_{\lambda(1)}}X_{u_{\lambda(3)}}+X^2_{u_{\lambda(1)}}+1.
$$
Hence, we have
$$
2r^2_2=r_4+r_1r_3+r^2_1+1.
$$
However, The equations \eqref{induction} tell us that
$r^2_1=r_2+1$ and $r_4+r_2=r_1r_3$. Therefore, we have
$r_2^2=r_1r_3+1.$ It agrees with Proposition \ref{ARcluster}.
\\
\\
\nd {\textbf{Acknowledgements.}} I am grateful to Dr. Xueqing
Cheng and Prof. Bin Zhu for helpful discussions.
\bibliographystyle{amsplain}

\end{document}